\documentclass[a4paper,reqno,11pt]{amsart}

\usepackage{amsmath,amsfonts,amssymb}
\usepackage{graphics,color}
\usepackage[latin1]{inputenc}
\usepackage[colorlinks = true,
            linkcolor = blue,
            urlcolor  = blue,
            citecolor = blue,
            anchorcolor = blue]{hyperref}
\usepackage{mathrsfs}
\newtheorem{theorem}{Theorem}[section]
\newtheorem{lemma}[theorem]{Lemma}

\makeatletter
\@namedef{subjclassname@2020}{\textup{2020} Mathematics Subject Classification}
\makeatother

\numberwithin{equation}{section}
\newtheorem{proposition}[theorem]{Proposition}
\newtheorem{corollary}[theorem]{Corollary}

\newtheorem{remark}[theorem]{Remark}

\newcommand{\Om}{\Omega}
\newcommand{\Gm}{\Gamma}
\newcommand{\pa}{\partial}

\newcommand{\dist}{\operatorname{dist}}

\newcommand{\Span}{\operatorname{span}}
\newcommand{\inner}[2]{\langle #1,#2\rangle}
\newcommand{\norm}[1]{\lVert #1\rVert}

\newcommand{\da}{\delta_a}
\newcommand{\ea}{\varepsilon_a}

\usepackage{cite}
\numberwithin{equation}{section}

\allowdisplaybreaks[4]
\begin{document}
	
\title[Attractive Bose-Einstein condensates with a ring-shaped potential]
{On the location of mass concentration for attractive Bose-Einstein condensates with a ring-shaped potential in bounded domains}

\author{Shubin Yu}
\address{Chongqing Key Laboratory of Group \& Graph Theories and Applications, Chongqing University of
Arts and Sciences, Chongqing 402160, People's Republic of China.}
\email{yshubin168@email.swu.edu.cn}

\author{Chun-Lei Tang$^*$}
\address{School of Mathematics and Statistics, Southwest University, Chongqing 400715, People's Republic of China.}
\email{tangcl@swu.edu.cn}

\author{Ziheng Zhang}
\address{School of Mathematical Sciences, Tiangong University, Tianjin 300387, People's Republic of China.}
\email{zhzh@mail.bnu.edu.cn}

\footnotetext[1]{Corresponding author.}
\footnotetext[2]{Project supported by the National Natural Science Foundation of China (No.12371120).}


\begin{abstract}
In this paper, we study the ground states of attractive
Bose-Einstein condensates in a bounded domain $\Omega\subset \mathbb R^2$ under the
ring-shaped trapping potential
$$
V(x)=(|x|-1)^2.
$$
Letting $\Gamma=\{x\in\mathbb R^2:|x|=1\}$ be the ring of global minima of $V(x)$, we focus on a unified geometric setting $\Gamma\cap \Omega\not=\emptyset$, which allows $\Gamma\subset \Omega$ or $\Gamma\cap\partial\Omega\neq\emptyset$.
In this setting, we determine the precise location of mass concentration for ground states as the interaction strength approaches the critical value.
More precisely,
we prove that the mass
concentrates at a point $x_0\in\Gamma\cap \Omega$ that has maximal distance from $\partial\Omega$. Our method relies on exponential upper and lower bounds on the difference between the minimal energies in the bounded domain and whole space. The essential lower bound follows from a uniform quadratic-growth estimate near
the ground state manifold of the whole space problem, which is established by combining the nondegeneracy of the optimizer for the Gagliardo-Nirenberg inequality with a Lyapunov-Schmidt reduction.
\end{abstract}
\maketitle
\vspace {-1cm}

\section{Introduction and main results}

Bose-Einstein condensation is a quantum phenomenon in which a macroscopic fraction of bosons occupies the same single-particle state below a critical temperature.
 Since its experimental realization in cold atomic gases
\cite{Anderson1995,Dalfovo1999}, the Gross-Pitaevskii (GP) theory has become a
standard mean-field description of trapped condensates, see for example
\cite{Gross1963,Gross1961,Pitaevskii1961,BaoCai2013}.
The forces
between the atoms in the condensates can be either attractive or repulsive.
In the attractive case, the system collapses if the particle number increases
beyond a critical value, see
\cite{Sackett1998,GuoSeiringer2014}.

To investigate the details of this collapse, Guo and Seiringer \cite{GuoSeiringer2014} studied the following GP functional
\begin{equation}\label{eqn:intro-GP-energy}
E_a^\infty(u)
:=
\int_{\mathbb R^2}\bigl(|\nabla u(x)|^2+V(x)|u(x)|^2\bigr)\,\mathrm dx
-\frac a2\int_{\mathbb R^2}|u(x)|^4\,\mathrm dx,\ u\in H_V^1(\mathbb R^2),
\end{equation}
where \(a>0\) describes the strength of the attractive interaction, \(V\)
is the trapping potential and
\[
H_V^1(\mathbb R^2)
:=
\left\{
u\in H^1(\mathbb R^2):
\int_{\mathbb R^2}V(x)|u(x)|^2\mathrm dx<\infty
\right\}
\]
endowed with the norm
\[
\|u\|_{H_V^1(\mathbb R^2)}^2
:=
\int_{\mathbb R^2}
\left(|\nabla u(x)|^2+(1+V(x))|u(x)|^2\right)\mathrm dx.
\]
The ground states of attractive Bose-Einstein condensates in $\mathbb R^2$ are  described by the
 minimizers of \eqref{eqn:intro-GP-energy} under the unit mass constraint $\|u\|_{L^2(\mathbb R^2)}=1.$
By the classical Gagliardo-Nirenberg inequality
\begin{equation}\label{eqn:GN}
\int_{\mathbb R^2}|u(x)|^4\mathrm dx
\le
\frac2{a^*}
\int_{\mathbb R^2}|\nabla u(x)|^2\mathrm dx
\int_{\mathbb R^2}|u(x)|^2\mathrm dx,
\end{equation}
they determined that $a^*>0$ is the critical value
 such that the minimizers exist if and only if $0<a<a^*$.
Here $a^*:=\|Q\|_{L^2(\mathbb R^2)}^2$, $Q$ is the
optimizer for the Gagliardo-Nirenberg inequality and the
unique positive solution (up to translations) of the scalar field equation
\[
-\Delta Q+Q-Q^3=0
\ \text{in }\mathbb R^2,
\]
see \cite{Weinstein1983}.
Moreover, they analyzed the mass concentration of minimizers as \(a\nearrow a^*\)
for the trapping potential
\begin{equation}\label{eqn:intro-product-potential}
V(x)
=
h(x)\prod_{i=1}^n|x-x_i|^{p_i},
\
p_i>0,\
C<h(x)<C^{-1},
\end{equation}
and showed that the mass concentrates at one of the flattest global minima of \(V(x)\).

Ring traps also arise naturally in experiments on
Bose-Einstein condensates, see for example
\cite{Ryu2007,HalkyardJonesGardiner2010}.  Compared to \eqref{eqn:intro-product-potential}, the present case has a non-discrete set of
global minima for the trapping potential $V(x)$.  For the ring-shaped potential
\begin{equation*}
V_A(x)=(|x|-A)^2,\  A>0,
\end{equation*}
Guo, Zeng and Zhou \cite{GuoZengZhou2016} proved that each minimizer concentrates at a point on the circular
bottom \(\{|x|=A\}\) and that symmetry breaking occurs as
\(a\nearrow a^*\).
Related results for attractive Bose-Einstein condensates and GP functionals in
$\mathbb R^2$ include \cite{GuoZhou2019,GuoWangZengZhou2018}
 for ellipse-shaped and multi-well
potentials, \cite{GuoLinWei2017} for the
local uniqueness and refined spike profiles, and \cite{PengYangZeng2024} for uniqueness modulo
rotations in the ring-shaped setting.

In what follows, we leave the whole space setting and turn to domains with boundary, where boundary geometry influences the location of mass concentration. An early study of
this phenomenon is the work of \cite{GuoLuoZhang2018}, where the authors considered
 the mass-critical Hartree energy functional.
For a bounded domain \(\Omega\) in
\(\mathbb R^4\), they analyzed the refined concentration behavior of
\(L^2\)-constrained minimizers and showed that the mass concentrates either
at an interior point or near the boundary, depending on whether the trapping
potential \(V(x)\) attains its flattest global minimum at an interior point of \(\Omega\) or not.
For the GP functional in bounded domains, similar results were obtained in
\cite{LuoZhu2019,YangTang2025}. In the ellipse-shaped setting, Guo, Yan and Zhao
\cite{GuoYanZhao2025} distinguished the cases where the relevant endpoints
of the major axis lie in the interior or on the boundary.
On the other hand, we mention that
Deng, Guo and Xu \cite{DengGuoXu2021} studied the GP functional in exterior domains, which can describe the
trapped
attractive Bose-Einstein condensates passing an obstacle in the plane.
More precisely, they gave the limit behavior of minimizers when the trapping potential $V(x)$ attains its global minima only along the whole
boundary, and proved that the mass concentration occurs at the flattest critical point of $V(x)$ on the boundary.
Li and Wang \cite{LiWang2025} investigated
potentials whose global minima occur both on the boundary and at an interior
point, such as $V(x)=(|x|-r)^2|x-y_0|^\alpha$ with $\alpha>0$, $\Omega=B_r(0)$ and $y_0\in \Omega$,
and obtained the blow-up criterion depending on whether $\alpha<2$.

Motivated by the above works, we
are interested in the ring-shaped potential
\begin{equation}\label{eqn:intro-our-ring}
V(x)=(|x|-1)^2
\end{equation}
in a bounded domain \(\Omega\subset\mathbb R^2\).
Define the ring of global minima of $V(x)$ by $\Gamma:=\{x\in\mathbb R^2:|x|=1\}.$
If $\Gamma\subset \partial\Omega$, the concentration behavior
of minimizers for the GP functional in $\Omega$
is of boundary type, as described in \cite{DengGuoXu2021}.
 If $\Gamma\subset \Omega$, then, as in the whole space problem \cite{GuoZengZhou2016}, interior concentration occurs, but the exact location is difficult to determine because all points of \(\Gamma\) have
the same local behavior. In addition, a
case of particular interest is
the geometric setting
\begin{equation}\label{eqn:intro-mixed-geometry}
\Gamma\cap\Omega\neq\emptyset,
\
\Gamma\cap\partial\Omega\not=\emptyset,
\end{equation}
where the ring $\Gamma$ intersects both the interior and
 the boundary of $\Omega$. To the best of our knowledge, this case has not been considered in previous studies, and
 the existing methods do not directly determine the
 location of mass concentration or even whether it occurs in the interior of $\Omega$.  Indeed, using the Gagliardo-Nirenberg
inequality with a remainder \cite{CarlenFrankLieb2014}, one can  establish an exponential estimate involving
the distance from the maximum point of minimizers to the boundary.
This estimate is effective
 for treating the case where
 all minima of $V(x)$ are discrete, lie entirely on the boundary or are of mixed type, see \cite{LiLuo2021,DengGuoXu2021,LiWang2025}.
However, the
estimate does not
provide the upper and lower bound estimates needed to
distinguish among different
interior or boundary points of $\Gamma$ if $V(x)$ satisfies \eqref{eqn:intro-our-ring} and \eqref{eqn:intro-mixed-geometry}.
This raises a natural question:
can the location of mass concentration be precisely characterized
if $\Gamma\subset \Omega$ or $\Gamma$ satisfies \eqref{eqn:intro-mixed-geometry}?

In this paper, we provide a clear answer to this question
under the unified geometric setting
$\Gamma\cap\Omega\neq\emptyset$, which allows $\Gamma\subset \Omega$ or $\Gamma\cap\partial\Omega\neq\emptyset$.
As we will see, the location of mass
concentration is precisely characterized by
 distances of the points on $\Gamma$ to $\partial \Omega$.
Let \(\Omega\subset\mathbb R^2\) be a bounded
connected domain with \(C^2\) boundary and let \(V\) be given by
\eqref{eqn:intro-our-ring}.  For \(a\ge0\), define
\begin{equation}\label{eqn:intro-bounded-energy}
E_a^\Omega(u)
:=
\int_\Omega\bigl(|\nabla u(x)|^2+V(x)|u(x)|^2\bigr)\,\mathrm dx
-\frac a2\int_\Omega|u(x)|^4\,\mathrm dx
\end{equation}
and
\begin{equation}\label{eqn:intro-bounded-minimum}
e_\Omega(a)
:=
\inf_{\{u\in H_0^1(\Omega): \|u\|_{L^2(\Omega)}=1\}}E_a^\Omega(u).
\end{equation}
By \cite{LiLuo2021,DengGuoXu2021}, the same critical value  \(a^*=\|Q\|_{L^2(\mathbb R^2)}^2\) as in the whole space problem \eqref{eqn:intro-GP-energy} arises.
For \(x\in\overline\Omega\), define
\[
d_\Omega(x):=\operatorname{dist}(x,\partial\Omega),
\]
and
\[
D_\Omega
:=
\max_{x\in\Gamma\cap\overline\Omega}d_\Omega(x),
\]
where $\mathcal M_\Omega:=\left\{
x\in\Gamma\cap\overline\Omega:
d_\Omega(x)=D_\Omega
\right\}$.
Under the assumption \(\Gamma\cap\Omega\neq\emptyset\), one has
\(D_\Omega>0\), and hence
\(\mathcal M_\Omega\subset\Gamma\cap\Omega\).

Our main results are as follows.
Set
\[
\lambda_0^4
:=
\frac12\int_{\mathbb R^2}|x|^2Q^2(x)\,\mathrm dx
\ \text{and}\
\delta_a
:=
\frac{(a^*-a)^\frac{1}{4}}{\lambda_0}.
\]
We first establish the existence and blow-up behavior of minimizers for $e_\Omega(a)$, and then
describe the precise location of mass concentration for minimizers.

\begin{theorem}\label{thm:main-1}
Assume that \(\Omega\subset\mathbb R^2\)  is a bounded
connected domain with \(C^2\) boundary and
\[
\Gamma\cap\Omega\neq\emptyset.
\]
\begin{enumerate}
\item [$(i)$]
For every \(0\le a<a^*\), \(e_\Omega(a)\) admits a nonnegative
minimizer.  Moreover, for $a\geq a^*$, \(e_\Omega(a)\) is not attained.
\item[$(ii)$]
Let \(u_a\)  be a nonnegative minimizer of \(e_\Omega(a)\) and \(x_a\) be a
 maximum point of \(u_a\).  Then,
 for any sequence $\{a_k\}$ satisfying
 \(a_k\nearrow a^*\) as $k\to\infty$, there exists a subsequence, still
denoted by $\{a_k\}$, such that each $u_{a_k}$ has a unique maximum point $x_{a_k}$ and
\[
\delta_{a_k} u_{a_k}(x_{a_k}+\delta_{a_k} x)
\to
\frac{Q(x)}{\sqrt{a^*}}
\ \text{in }H^1(\mathbb R^2)\ \text{as }k\to\infty.
\]
 Furthermore,
\[
\frac{|x_{a_k}|-1}{\delta_{a_k}}\to0,
\]
\[
e_\Omega({a_k})
=
\frac{2\lambda_0^2}{a^*}(a^*-{a_k})^{\frac{1}{2}}
+
o\bigl((a^*-{a_k})^{\frac{1}{2}}\bigr)
\]
and $x_{a_k}\to x_0\in \Gamma$ as $k\to\infty$.
\end{enumerate}
\end{theorem}

\begin{theorem}\label{thm:main}
Under the setting in Theorem {\rm\ref{thm:main-1}}, there holds
\[
d_\Omega(x_{a_k})\to D_\Omega.
\]
Consequently, $x_0\in \mathcal M_\Omega$ and
\[
\operatorname{dist}(x_{a_k},\mathcal M_\Omega)\to0.
\]

\end{theorem}

\begin{remark}\label{rem:selection-mechanism}\rm
Theorems~\ref{thm:main-1}-\ref{thm:main} show that although the points on \(\Gamma\) have same local behavior, mass concentration occurs only at points in $\Gamma\cap\Omega$ having maximal distance from $\partial\Omega$.
A closely related phenomenon is classical for singularly perturbed semilinear Dirichlet problems.  Ni and Wei
\cite{NiWei1995} proved that the spike of a least-energy solution approaches a point of
maximal distance from the boundary.  Further results involving local maxima
of the distance function were obtained in
\cite{Wei1996,LiNirenberg1998}.
\end{remark}

\begin{remark}\label{rem:method}\rm
Inspired by \cite{NiWei1995,Wei1996,LiNirenberg1998},
our proof of Theorem~\ref{thm:main} is based on a direct comparison between the bounded domain minimum
\(e_\Omega(a)\) and the corresponding whole space minimum \(e_\infty(a)\).
More precisely, we establish exponential upper and lower bounds for
$e_\Omega(a)-e_\infty(a)$.
The upper bound is obtained by cutting off a suitably rotated whole space minimizer.  The lower bound requires a quantitative comparison of a
bounded domain minimizer with the manifold $\mathcal G_a$ of whole space minimizers.
For this purpose, we prove in Section \ref{sec:Nondegeneracy} a uniform quadratic-growth estimate for
the whole space energy near that manifold.  Its proof uses the nondegeneracy of
the optimizer \(Q\)  and a
finite-dimensional Lyapunov-Schmidt reduction.
\end{remark}

As a direct application of Theorem~\ref{thm:main}, we present two explicit domains
corresponding to the cases $\Gamma\subset\Omega$ and $\Gamma\cap\partial\Omega\neq\emptyset$, respectively.
In both cases, $\mathcal M_\Omega=\{(1,0)\}$, so the location of mass concentration is uniquely determined.
\begin{corollary}
Let
\[
\Omega=B_2((1/2,0))\ \text{or}\ \Omega=B_2((1,0)),
\]
then the mass of all minimizers for \(e_\Omega(a)\) concentrates at \((1,0)\) as
\(a\nearrow a^*\).
\end{corollary}

The remainder of the paper is organized as follows. In Section~\ref{sec:Blow-up}, we establish the blow-up profile of the minimizers for \eqref{eqn:intro-bounded-minimum}, and we also derive some useful estimates for the corresponding whole space problem. These results yield Theorem~\ref{thm:main-1}. Section~\ref{sec:Nondegeneracy} is devoted to investigating the constrained nondegeneracy of the limiting profile and the uniform quadratic-growth estimate for the whole space energy near the ground state manifold. In Section~\ref{sec:boundary-correction}, we derive exponential upper and lower bounds for the difference between the bounded domain and whole space minimal energies. In Section \ref{sec:Proof-main}, we combine these estimates to prove Theorem~\ref{thm:main}.

\section{Blow-up analysis}\label{sec:Blow-up}

This section provides the blow-up results used throughout the paper.  First,
we give the blow-up profile of the minimizers for \eqref{eqn:intro-bounded-minimum}.
Then we study the corresponding whole space problem and derive uniform
pointwise decay and lower exponential estimates.
The two problems have the same blow-up scale, which is essential for the energy comparison in Section \ref{sec:boundary-correction}.
Before presenting our argument,
 we remark that the notations $c$ and $C$ in the context denote the positive constants,
which might be changed from line to line and even in the same line.
\begin{lemma}\label{lem:exist}
For $0\le a<a^*$, $e_\Om(a)$ admits a minimizer.  Moreover,
\begin{equation*}
e_\Om(a)
\le
\frac{2\lambda_0^2}{a^*}(a^*-a)^{1/2}
+
o\bigl((a^*-a)^{1/2}\bigr).
\end{equation*}
In addition, $e_\Om(a^*)=0$ is not attained, while $e_\Om(a)=-\infty$ for
$a>a^*$.
\end{lemma}
\begin{proof}
Since $\Gm\cap\Om\neq\emptyset$, the proof is similar to \cite{GuoSeiringer2014,GuoZengZhou2016,LuoZhu2019}, so we omit the details.
\end{proof}
\begin{lemma}\label{lem:blowup}
Let $u_a$ be a nonnegative minimizer and set
\begin{equation*}
\varepsilon_a:=
\left(\int_\Om|\nabla u_a(x)|^2\mathrm dx\right)^{-1/2}.
\end{equation*}
Then $\ea\to0$ as $a\nearrow a^*$. If $x_a$ is a maximum point of $u_a$ and
\begin{equation*}
w_a(x)
:=
\varepsilon_au_a(x_a+\varepsilon_ax)
\end{equation*}
in
\[
\Om_a:=\{x\in\mathbb R^2:x_a+\varepsilon_ax\in\Om\},
\]
extended by zero to $\mathbb R^2\setminus\Om_a$, then
for any sequence $\{a_k\}$ satisfying
 \(a_k\nearrow a^*\) as $k\to\infty$, there exists a subsequence, still
denoted by $\{a_k\}$, such that
\begin{equation*}
\frac{d_\Om(x_{a_k})}{\varepsilon_{a_k}}\to\infty,
\end{equation*}
\begin{equation}\label{eqn:conver-H1}
w_{a_k}\to q:=\frac{Q}{\sqrt{a^*}}
\ \hbox{ in }H^1(\mathbb R^2),
\end{equation}
and $x_{a_k}$ is the unique maximum point of $u_{a_k}$ satisfying
\begin{equation*}
\dist(x_{a_k},\Gm)\to0
\end{equation*}
as $k\to\infty$.
\end{lemma}
\begin{proof}
The proof is similar to \cite[Lemma 3.3]{DengGuoXu2021}, so we also omit the details.
\end{proof}
\begin{lemma}\label{lem:scale}
Define
\begin{equation*}
r_a:=\frac{|x_a|-1}{\ea}.
\end{equation*}
For the sequence $\{a_k\}$ obtained in
Lemma {\rm\ref{lem:blowup}}, there exists a subsequence, still
denoted by $\{a_k\}$, such that
 $\{r_{a_k}\}$ is uniformly bounded,
\begin{equation*}
r_{a_k}\to0,
\end{equation*}
\begin{equation*}
\frac{\varepsilon_{a_k}}{\delta_{a_k}}\to1,
\end{equation*}
and
\begin{equation*}
e_\Omega(a_k)
=
\frac{2\lambda_0^2}{a^*}(a^*-a_k)^{1/2}
+
o\bigl((a^*-a_k)^{1/2}\bigr)
\end{equation*}
as $k\to\infty$.
In particular, $x_{a_k}\to x_0\in\Gamma$ as $k\to\infty$.
\end{lemma}

\begin{proof} Define $P_{a_k}=\int_\Om V(x)u_{a_k}^2(x)\mathrm dx$.
We first prove that there exists a constant $c>0$ such that
\begin{equation}\label{P-a-low-bound}
P_{a_k}\ge c\varepsilon_{a_k}^2.
\end{equation}
If $|r_{a_k}|\to\infty$ as $k\to\infty$, then for every fixed $R>0$,
\[
\bigl||x_{a_k}+\varepsilon_{a_k} x|-1\bigr|
\ge \bigl||x_{a_k}|-1\bigr|-\bigl||x_{a_k}+\varepsilon_{a_k} x|-|x_{a_k}|\bigr|\ge \varepsilon_{a_k}|r_{a_k}|-\varepsilon_{a_k}|x|
\ge\varepsilon_{a_k}(|r_{a_k}|-R),
\  |x|\le R.
\]
Thus,
\begin{equation}\label{P-a-e-a-2-local}
\frac{P_{a_k}}{\varepsilon_{a_k}^2}
\ge
(|r_{a_k}|-R)^2\int_{B_R(0)}w_{a_k}^2(x)\mathrm dx\to\infty
\end{equation}
and \eqref{P-a-low-bound} holds.
If  $r_{a_k}$ is uniformly bounded, then up to a subsequence, $r_{a_k}\to r_0$ and
$x_{a_k}/|x_{a_k}|\to e_0\in \mathbb S^1$ as $k\to\infty$.  For $x\in B_R(0):=\{x\in\mathbb R^2: |x|<R\}$,
\[
\frac{|x_{a_k}+\varepsilon_{a_k} x|-1}{\varepsilon_{a_k}}
=
r_{a_k}+\frac{x_{a_k}}{|x_{a_k}|}\cdot x+O(\varepsilon_{{a_k}}|x|^2).
\]
Hence, by \eqref{eqn:conver-H1}, we conclude that
\begin{equation}\label{eqn:P-a-est}
\begin{aligned}
\liminf_{k\to\infty}\frac{P_{a_k}}{\varepsilon_{a_k}^2}
&\ge\liminf_{k\to\infty} \int_{|x|\leq R}\left(\frac{|x_{a_k}+\varepsilon_{a_k} x|-1}{\varepsilon_{a_k}}\right)^2w_{a_k}^2(x)\mathrm dx
\\
&=
\int_{B_R(0)}(r_0+e_0\cdot x)^2q^2(x)\mathrm dx.
\end{aligned}
\end{equation}
Letting $R\to\infty$, we have
\begin{equation}\label{P-a-low-bound-fine}
\liminf_{k\to\infty}\frac{P_{a_k}}{\varepsilon_{a_k}^2}\ge r_0^2+\frac{\lambda_0^4}{a^*}>0.
\end{equation}
This proves $P_{a_k}\ge c\varepsilon_{a_k}^2$, that is, \eqref{P-a-low-bound} holds.

Since
\[
\varepsilon_{a_k}^2\int_\Om u_{a_k}^4(x)\,\mathrm dx\to\frac{2}{a^*}
\]
as $k\to\infty$,
there is \(c_1>0\) such that
for all sufficiently large \(k\),
\[
\frac{a^*-{a_k}}{2}\int_\Om u_{a_k}^4(x)\,\mathrm dx
\ge
c_1\frac{a^*-{a_k}}{\varepsilon_{a_k}^2}.
\]
Combining \eqref{P-a-low-bound} and
Lemma~\ref{lem:exist}, there are constants \(c_1,c_2>0\) such that
\[
c_1\frac{a^*-{a_k}}{\varepsilon_{a_k}^2}+c_2\varepsilon_{a_k}^2
\le
C(a^*-{a_k})^{1/2},
\]
which indicates that
\begin{equation*}
c_1(a^*-{a_k})^{1/2}
\le
\varepsilon_{a_k}^2
\le
c_2(a^*-a_k)^{1/2}.
\end{equation*}
Since $P_{a_k}\le C(a^*-a_k)^{1/2}$, one has
$P_{a_k}/\varepsilon_{a_k}^2\le C$, and the preceding local lower bound \eqref{P-a-e-a-2-local}
implies that $\{r_{a_k}\}$
 is uniformly bounded.
Take any subsequence of $\{a_k\}$ such that
\[
r_{a_k}\to r_0,\
t_{a_k}:=\frac{\varepsilon_{a_k}^2}{(a^*-{a_k})^{1/2}}\to t_0\in(0,\infty).
\]
By \eqref{P-a-low-bound-fine}, we infer that
\begin{equation*}
\liminf_{k\to\infty}
\frac{e_\Om({a_k})}{(a^*-{a_k})^{1/2}}
\ge
\frac1{a^*t_0}
+
t_0\left(
r_0^2+\frac{\lambda_0^4}{a^*}
\right).
\end{equation*}
On the other hand,
 Lemma~\ref{lem:exist} gives
\[
\limsup_{k\to\infty}
\frac{e_\Om({a_k})}{(a^*-{a_k})^{1/2}}
\le
\frac{2\lambda_0^2}{a^*}.
\]
Since
\[
\frac1{a^*t}+\frac{\lambda_0^4}{a^*}t
\ge
\frac{2\lambda_0^2}{a^*}
\]
with equality only at $t=\lambda_0^{-2}$, we obtain
$r_0=0$ and $t_0=\lambda_0^{-2}$.
This completes the proof.
\end{proof}

 Next, we turn to the corresponding whole space problem.
Define
\begin{equation*}
e_\infty(a)
:=
\inf_{\{u\in H_V^1(\mathbb R^2):\|u\|_{L^2(\mathbb R^2)}=1\}}
E_a^\infty(u),
\end{equation*}
where
\[
E_a^\infty(u)
=
\int_{\mathbb R^2}\bigl(|\nabla u(x)|^2+V(x)u^2(x)\bigr)\mathrm dx
-\frac a2\int_{\mathbb R^2}u^4(x)\mathrm dx.
\]
Let $\mathcal G_a$
denote the set of all ground states of the
whole space problem, namely, all positive minimizers of
$e_\infty(a)$. For \(a\) sufficiently close to \(a^*\), the minimizer is unique up to
rotations around the origin, see \cite{PengYangZeng2024,zeng2025}.
For $U_a\in \mathcal G_a$,
the corresponding Euler-Lagrange equation is
\[
-\Delta U_a+VU_a
=
\mu_a^\infty U_a+aU_a^3
\ \text{in }\mathbb R^2,
\]
where $\mu_a^\infty $ is the Lagrange multiplier.

\begin{lemma}\label{lem:whole}
As \(a\nearrow a^*\),
\[
e_\infty(a)
=
\frac{2\lambda_0^2}{a^*}(a^*-a)^{1/2}
+
o\bigl((a^*-a)^{1/2}\bigr).
\]
For any sequence $\{a_k\}$ satisfying
 \(a_k\nearrow a^*\) as $k\to\infty$, there exists a subsequence, still
denoted by $\{a_k\}$, such that \(U_{a_k}\) has a unique  maximum
point $p_{a_k}$ and as \(k\to\infty\),
\[
-\mu_{a_k}^\infty\delta_{a_k}^2\to1,
\
\frac{|p_{a_k}|-1}{\delta_{a_k}}\to0,
\
\delta_{a_k}^2
\int_{\mathbb R^2}U_{a_k}^4(x)\,\mathrm dx
\to\frac{2}{a^*},
\]
and
\[
W_{a_k}(x)
:=
\delta_{a_k}
U_{a_k}(p_{a_k}+\delta_{a_k}x)
\to q(x)
\ \text{ in }
H^1(\mathbb R^2)\cap C_{\rm loc}^2(\mathbb R^2).
\]
Moreover, for every \(\eta\in(0,1/4)\), there
exist constants \(R_\eta,C_\eta>0\), independent of \(k\), such that for all sufficiently
large \(k\),
\begin{equation}\label{eqn:W-A-decay}
W_{a_k}(x)
+
|\nabla W_{a_k}(x)|
+
|D^2W_{a_k}(x)|
\le
C_\eta e^{-(1-\eta)|x|},
\  |x|\ge R_\eta.
\end{equation}
In addition, for every integer \(m\ge0\),
\begin{equation}\label{eqn:weighted-L-2-bound}
\sum_{|\alpha|\le3}
\left\|
(1+|x|)^mD^\alpha W_{a_k}
\right\|_{L^2(\mathbb R^2)}
<\infty
\end{equation}
uniformly for all sufficiently large \(k\), where \(\alpha\) is a multi-index.
Finally, for every fixed \(L>0\), there exists \(c_{\eta,L}>0\),
independent of \(k\), such that
\[
W_{a_k}(x)
\ge
c_{\eta,L}e^{-(1+\eta)|x|}
\]
when
\[
R_\eta\le |x|\le \frac{L}{\delta_{a_k}}
\]
and \(k\) is sufficiently large.
\end{lemma}

\begin{proof}
The energy estimate and the blow-up results follow from
\cite{GuoZengZhou2016}. Moreover, we recall from \cite{GuoZengZhou2016} that $W_{a_k}(x)\to 0$ as $|x|\to\infty$ uniformly in $k$. Now, we give the proof of decay estimates.
Note that the rescaled function \(W_{a_k}\) satisfies
\begin{equation}\label{eqn:rescaled-equation}
-\Delta W_{a_k}
+
\left[
-\mu_{a_k}^\infty\delta_{a_k}^2
+
\delta_{a_k}^2
V(p_{a_k}+\delta_{a_k}x)
\right]W_{a_k}
=
a_kW_{a_k}^3
\ \text{in }\mathbb R^2.
\end{equation}
Then for each fixed
\(\eta\in(0,1/4)\), we can choose \(R_\eta>0\) such that
\[
-\mu_{a_k}^\infty\delta_{a_k}^2
+
\delta_{a_k}^2V(p_{a_k}+\delta_{a_k}x)
-
a_kW_{a_k}^2(x)
\ge
\left(1-\frac{\eta}{4}\right)^2
\]
for \(|x|\ge R_\eta\) and all sufficiently large \(k\).
Set
$\sigma:=1-\frac{\eta}{4}$,
then
\[
-\Delta W_{a_k}+\sigma^2W_{a_k}\le0,\
|x|\ge R_\eta.
\]
On the other hand,
\[
(-\Delta+\sigma^2)e^{-\sigma|x|}
=
\frac{\sigma}{|x|}e^{-\sigma|x|}
>0,
\  x\ne0.
\]
Since \(W_{a_k}\to q\) locally uniformly, replacing $C_\eta$ by a larger constant if necessary, we have
\[
W_{a_k}(x)
\le
C_\eta e^{-\sigma R_\eta}
\
\text{on }\partial B_{R_\eta}(0)
\]
uniformly for all sufficiently large \(k\). Therefore the comparison principle
 gives
\[
W_{a_k}(x)
\le
C_\eta e^{-\sigma|x|},
\  |x|\ge R_\eta.
\]

We next estimate the spatial derivatives. Let
\(
-\Delta W_{a_k}=F_{a_k},
\)
where
\[
F_{a_k}(x)
:=
\mu_{a_k}^\infty\delta_{a_k}^2W_{a_k}
+
a_kW_{a_k}^3
-
\delta_{a_k}^2
V(p_{a_k}+\delta_{a_k}x)W_{a_k}.
\]
Since
\[
\frac{|p_{a_k}|-1}{\delta_{a_k}}\to0\ \text{as}\ k\to\infty,
\]
the sequence \(\{p_{a_k}\}\) is bounded. Obviously, there exists $C>0$ such that
$V(x)\le C(1+|x|^2)$
for $x\in\mathbb R^2$.
Hence, for every fixed \(p>2\), the preceding exponential estimate
implies that there exists $C_{\eta,p}>0$ such that
for \(r:=|x|\ge R_\eta+2\) and some fixed integer \(M\ge0\),
\[
\|F_{a_k}\|_{L^p(B_1(x))}
+
\|W_{a_k}\|_{L^p(B_1(x))}
\le
C_{\eta,p}(1+r)^M e^{-\sigma(r-1)}
\]
uniformly for all sufficiently large
\(k\). The standard interior \(W^{2,p}\)-estimate
\cite{Gilbarg-1997} yields for $p>2$,
\[
\|W_{a_k}\|_{W^{2,p}(B_{1/2}(x))}
\le
C_{\eta,p}(1+r)^M e^{-\sigma(r-1)}.
\]
Since the space dimension is two,
\[
W^{2,p}(B_{1/2}(x))
\hookrightarrow
C^{1,\alpha}(B_{1/2}(x)),
\
\alpha=1-\frac2p,
\]
and then
\[
|\nabla W_{a_k}(x)|
\le
C_{\eta,p}(1+r)^M e^{-\sigma(r-1)}.
\]

To estimate the second derivatives, differentiate the rescaled equation \eqref{eqn:rescaled-equation}
in the distributional sense. Since
\(
V\in W^{1,\infty}_{\rm loc}(\mathbb R^2),
\) we obtain for \(j=1,2\),
\[
\begin{aligned}
-\Delta(\partial_jW_{a_k})
=&
\left(
\mu_{a_k}^\infty\delta_{a_k}^2
+
3a_kW_{a_k}^2
-
\delta_{a_k}^2
V(p_{a_k}+\delta_{a_k}x)
\right)
\partial_jW_{a_k}
\\
&-
\delta_{a_k}^3
(\partial_jV)(p_{a_k}+\delta_{a_k}x)
W_{a_k}.
\end{aligned}
\]
Here $\pa_j:=\frac{\partial}{\partial x_j}$, $j=1,2$.
Note that
\(
|\nabla V(x)|
\le
C(1+|x|)
\ \text{for a.e. }x\in\mathbb R^2.
\)
Using the bounds already obtained for \(W_{a_k}\) and
\(\nabla W_{a_k}\), we get
\[
\begin{aligned}
&
\left\|
\left(
\mu_{a_k}^\infty\delta_{a_k}^2
+
3a_kW_{a_k}^2
-
\delta_{a_k}^2
V(p_{a_k}+\delta_{a_k}\,\cdot)
\right)
\partial_jW_{a_k}
\right\|_{L^p(B_1(x))}
\\
&\
+
\left\|
\delta_{a_k}^3
(\partial_jV)(p_{a_k}+\delta_{a_k}\,\cdot)
W_{a_k}
\right\|_{L^p(B_1(x))}
\\
&\
\le
C_{\eta,p}(1+r)^M e^{-\sigma(r-1)}.
\end{aligned}
\]
Applying the interior \(W^{2,p}\)-estimate once more, now to
\(\partial_jW_{a_k}\), we have
\[
\|\partial_jW_{a_k}\|_{W^{2,p}(B_{1/2}(x))}
\le
C_{\eta,p}(1+r)^M e^{-\sigma(r-1)}.
\]
Consequently,
\[
|D^2W_{a_k}(x)|
\le
C_{\eta,p}(1+r)^M e^{-\sigma(r-1)}.
\]
Since
\(
\sigma-(1-\eta)=\frac{3\eta}{4}>0,
\)
then
\eqref{eqn:W-A-decay} holds. Moreover, by the standard interior
higher-regularity estimates (see, e.g.,
\cite[Section~6.3]{Evans-2010}), we have
\[
\|D^3W_{a_k}\|_{L^2(B_{1/2}(x))}
\le
C_\eta(1+r)^M e^{-\sigma(r-1)}
\]
for \(r=|x|\ge R_\eta+1\), uniformly for all sufficiently large \(k\). Thus,
\eqref{eqn:weighted-L-2-bound} holds for every integer \(m\ge0\).

Finally, we prove the lower exponential estimate. Fix \(L>0\) and choose
\(
1<\kappa<\lambda<1+\eta.
\)
Taking \(R_\eta\) sufficiently large,
\[
\kappa^2-\lambda^2+\frac{\lambda}{r}<0,
\  r\ge R_\eta.
\]
In the annulus
$
R_\eta
\le
|x|
\le
\frac{2L}{\delta_{a_k}},
$
the point $
p_{a_k}+\delta_{a_k}x
$
remains in a fixed bounded set depending only on \(L\). Hence, as $k\to\infty$,
\[
\sup_{|x|\le2L/\delta_{a_k}}
\delta_{a_k}^2
V(p_{a_k}+\delta_{a_k}x)
\to0.
\]
Together with
$
-\mu_{a_k}^\infty\delta_{a_k}^2\to1
$
and \(\kappa>1\), this means that for all sufficiently large \(k\),
\[
-\mu_{a_k}^\infty\delta_{a_k}^2
+
\delta_{a_k}^2
V(p_{a_k}+\delta_{a_k}x)
\le
\kappa^2
\]
in the annulus $R_\eta
\le
|x|
\le
\frac{2L}{\delta_{a_k}}
$, and then
\[
-\Delta W_{a_k}+\kappa^2W_{a_k}\ge0.
\]

Set
\[
\psi_{a_k}(r)
:=
e^{-\lambda r}
-
e^{-\lambda\frac{2L}{\delta_{a_k}}},
\]
then
$
\psi_{a_k}\left(\frac{2L}{\delta_{a_k}}\right)=0.
$
For
$
R_\eta\le r\le\frac{2L}{\delta_{a_k}},
$
we have
\[
(-\Delta+\kappa^2)\psi_{a_k}(r)
=
\left(
\kappa^2-\lambda^2+\frac{\lambda}{r}
\right)e^{-\lambda r}
-
\kappa^2e^{-\lambda\frac{2L}{\delta_{a_k}}}
<0.
\]
Because \(W_{a_k}\to q\) locally uniformly and \(q>0\), there exists
\(c_0>0\) such that
\[
W_{a_k}(x)\ge c_0,
\
|x|=R_\eta
\]
for all sufficiently large \(k\). Choosing \(c>0\) sufficiently small,
we have
\[
W_{a_k}\ge c\psi_{a_k}
\ \text{on }|x|=R_\eta,
\]
while on the outer boundary \(|x|=\frac{2L}{\delta_{a_k}}\), we have $c\psi_{a_k}=0\le W_{a_k}$.
The comparison principle therefore yields
\[
W_{a_k}(x)
\ge
c\psi_{a_k}(|x|),
\
R_\eta
\le
|x|
\le
\frac{2L}{\delta_{a_k}}.
\]
If $|x|\le\frac{L}{\delta_{a_k}}$,
then
$
\frac{2L}{\delta_{a_k}}-|x|
\ge
\frac{L}{\delta_{a_k}},
$
and, for all sufficiently large \(k\),
\[
\begin{aligned}
\psi_{a_k}(|x|)
&=
e^{-\lambda|x|}
\left[
1-
e^{-\lambda
\left(
\frac{2L}{\delta_{a_k}}-|x|
\right)}
\right]
\ge
\frac12e^{-\lambda|x|}
\ge
\frac12e^{-(1+\eta)|x|}.
\end{aligned}
\]
This completes the proof.
\end{proof}

\section{Nondegeneracy and a uniform quadratic-growth estimate}\label{sec:Nondegeneracy}
The purpose of this section is to establish the uniform quadratic-growth estimate for the whole space energy near the ground state manifold \(\mathcal{G}_a\), which will provide the key lower-bound ingredient in the comparison between \(e_{\Omega}(a)\) and \(e_{\infty}(a)\) in Section \ref{sec:boundary-correction}. At a whole space minimizer \(U_a\in\mathcal{G}_a\), the direct linearization of \(E_a^\infty\) on the mass constraint yields the quadratic form
\[
\mathfrak{q}_{a,U_a}(h,h)
:=
\int_{\mathbb{R}^2}
\left[
|\nabla h|^2
+
(V-\mu_a^\infty)h^2
-
3aU_a^2h^2
\right]\mathrm{d}x,
\
h\in T_{U_a},
\]
where
\[
T_{U_a}=\{h\in H_V^1(\mathbb{R}^2):\langle U_a,h\rangle_{L^2(\mathbb R^2)}=0\}.
\]
If \(\mathfrak{q}_{a,U_a}\) is uniformly coercive on \(T_{U_a}\), it would immediately give a quadratic lower bound for the energy near \(U_a\).

However, this is not the case: the ring-shaped potential generates an exact rotational degeneracy, and the translation and scaling directions become almost degenerate as \(a\nearrow a^*\). Consequently, the quadratic form \(\mathfrak{q}_{a,U_a}\) has zero or nearly zero directions, and no uniform positive lower bound can hold on the tangent space \(T_{U_a}\). We therefore first prove the constrained nondegeneracy of the limiting profile \(Q\), identifying the zero directions of the linearized form under the mass constraint.
Then near a whole space minimizer \(U_a\), we remove the rotational degeneracy by an exact modulation
and treat the remaining radial-translation and scaling directions through a two-dimensional Lyapunov-Schmidt reduction. This yields uniform coercivity of order \(\delta_a^{-2}\) on the complementary space and a positive quadratic lower bound for the reduced energy. Finally, we transfer the weighted \(\#\)-norm estimate to the \(\ast\)-norm used later, thereby obtaining the desired uniform quadratic-growth estimate for \(E_a^\infty\) near \(\mathcal{G}_a\) for \(a\) sufficiently close to $a^*$.

\subsection{Constrained nondegeneracy of the limiting profile}
\ \\

Set
\begin{equation*}
\Lambda Q:=Q+x\cdot\nabla Q
\end{equation*}
and the quadratic form
\begin{equation*}
\mathcal B(\phi,\psi)
:=
\int_{\mathbb R^2}
\left[
\nabla\phi(x)\cdot\nabla\psi(x)
+
\phi(x)\psi(x)
-
3Q^2(x)\phi(x)\psi(x)
\right] \mathrm dx
=\langle L_+\phi,\psi\rangle.
\end{equation*}
For simplicity,
we write
$
\langle f,g\rangle
:=\langle f,g\rangle_{L^2(\mathbb R^2)}
$
for \(f,g\in L^2(\mathbb R^2)\).
Define the linearized
operator $L_+:=-\Delta+1-3Q^2$, then $\mathcal B(\phi,\psi)=\langle L_+\phi,\psi \rangle$.
To establish the following
 constrained nondegeneracy, we shall use the classical kernel characterization of the linearized
operator $L_+$, see \cite{Kwong1989,NiTakagi1991}.
Namely, $ker(L_+)=\Span\{\pa_1Q,\pa_2Q\}$, where $\pa_iQ=\frac{\partial Q}{\partial x_i}$, $i=1,2$.

\begin{lemma}\label{lem:Qcoercive}
If $\inner{Q}{\phi}=0$, then
\[
\mathcal B(\phi,\phi)\ge0.
\]
Moreover, the equality holds if and only if
\[
\phi\in
\Span\{\Lambda Q,\pa_1Q,\pa_2Q\}.
\]
Consequently, there exists $c_0>0$ such that
\begin{equation*}
\mathcal B(\phi,\phi)
\ge
c_0\norm{\phi}_{H^1(\mathbb R^2)}^2
\end{equation*}
when
\[
\inner{\phi}{Q}
=
\inner{\phi}{\Lambda Q}
=
\inner{\phi}{\pa_1Q}
=
\inner{\phi}{\pa_2Q}
=0.
\]
\end{lemma}

\begin{proof}
Consider
\[
F(u):=\int_{\mathbb R^2}|\nabla u(x)|^2\mathrm dx-\frac12\int_{\mathbb R^2}u^4(x)\mathrm dx
\]
on the sphere $\|u\|_{L^2(\mathbb R^2)}^2=a^*$.  By the Gagliardo-Nirenberg
inequality \eqref{eqn:GN}, $F(u)\ge F(Q)=0$.  If $\phi\perp Q$, define
\[
Q_t
:=
\frac{\sqrt{a^*}(Q+t\phi)}{\|Q+t\phi\|_{L^2(\mathbb R^2)}}
\]
for $t\in \mathbb R$.
Then
\[
Q_t
=
Q+t\phi
-\frac{t^2}{2a^*}\|\phi\|_{L^2(\mathbb R^2)}^2Q
+O(t^3),
\]
and
\[
F(Q_t)
=
t^2\mathcal B(\phi,\phi)+O(t^3).
\]
Hence $\mathcal B(\phi,\phi)\ge0$.

If equality holds, then for every $\psi\perp Q$,
$
\mathcal B(\phi,\psi)=0.
$
Indeed, for any $\varepsilon\in\mathbb R$, we have $\phi+\varepsilon\psi\perp Q$
and
\[
0\leq\mathcal B(\phi+\varepsilon\psi,\phi+\varepsilon\psi)=\mathcal B(\phi,\phi)+2\varepsilon
\mathcal B(\phi,\psi)+\varepsilon^2\mathcal B(\psi,\psi).
\]
If $\mathcal B(\phi,\psi)\not=0$, choosing $\varepsilon$  sufficiently small with an appropriate sign,
we obtain a contradiction. That is, $ \mathcal B(\phi,\psi)=\langle L_+\phi,\psi\rangle=0$, $\forall \psi\perp Q$,
and thus
$
L_+\phi=cQ
$
for some $c\in\mathbb R$.  Since
$
L_+(\Lambda Q)=-2Q,
$
one has
$
L_+\left(\phi+\frac c2\Lambda Q\right)=0.
$
Therefore the nondegeneracy of $Q$ gives
\[
\phi+\frac c2\Lambda Q
\in
\Span\{\pa_1Q,\pa_2Q\},
\]
that is,
$
\phi\in
\Span\{\Lambda Q,\pa_1Q,\pa_2Q\}.
$
The reverse is obvious.

Now, we prove the final coercivity estimate.
Suppose by contradiction that there is a sequence \(\{\phi_n\}\subset H^1(\mathbb R^2)\) such that
\[
\langle \phi_n,Q\rangle
=
\langle \phi_n,\Lambda Q\rangle
=
\langle \phi_n,\partial_1Q\rangle
=
\langle \phi_n,\partial_2Q\rangle
=0
\]
and
\[
\|\phi_n\|_{H^1(\mathbb R^2)}=1,\
\mathcal B(\phi_n,\phi_n)\to 0
\]
as $n\to\infty$.
Up to a subsequence, we may assume that
$
\phi_n\rightharpoonup \phi_0 \ \text{in}\ H^1(\mathbb R^2)
$
as $n\to\infty$.
Then
\begin{equation}\label{eqn:phi-0}
\langle \phi_0,Q\rangle=
\langle \phi_0,\Lambda Q\rangle=
\langle \phi_0,\partial_1Q\rangle=
\langle \phi_0,\partial_2Q\rangle=0
\end{equation}
and
\[
\int_{\mathbb R^2} Q^2(x)\phi_n^2(x)\mathrm dx
\to
\int_{\mathbb R^2} Q^2(x)\phi_0^2(x)\mathrm dx
\]
as $n\to\infty$.
On the other hand, from \(\|\phi_n\|_{H^1(\mathbb R^2)}=1\) and \(\mathcal B(\phi_n,\phi_n)\to 0\) as $n\to\infty$, we obtain that as $n\to\infty$
\[
\mathcal B(\phi_n,\phi_n)
=
\|\phi_n\|_{H^1(\mathbb R^2)}^2
-
3\int_{\mathbb R^2}Q^2(x)\phi_n^2(x)\mathrm dx
\to  0,
\]
which implies that
\[
3\int_{\mathbb R^2}Q^2(x)\phi_n^2(x)\mathrm dx\to 1\
\text{and}\
\int_{\mathbb R^2}Q^2(x)\phi_0^2(x)\mathrm dx=\frac13>0.
\]
This means that \(\phi_0\neq 0\).
By the weak lower semicontinuity of the \(H^1(\mathbb R^2)\)-norm, we can deduce that
\[
\mathcal B(\phi_0,\phi_0)
=
\|\phi_0\|_{H^1(\mathbb R^2)}^2
-
3\int_{\mathbb R^2}Q^2(x)\phi_0^2(x)\mathrm dx
\le0.
\]
On the other hand, \(\phi_0\perp Q\) in \(L^2(\mathbb R^2)\) states that
$
\mathcal B(\phi_0,\phi_0)\ge 0.
$
Hence $\mathcal B(\phi_0,\phi_0)=0$
and $\phi_0\in \operatorname{span}\{\Lambda Q,\partial_1Q,\partial_2Q\}.$
From \eqref{eqn:phi-0},
$\phi_0=0.$
This contradicts the fact that \(\phi_0\neq 0\).
The proof is completed.
\end{proof}

\subsection{Uniform quadratic-growth estimate}
\ \\

Fix \(U_a\in\mathcal G_a\) and recall that \(p_a\) is the global maximum point of
\(U_a\), \(\mu_a^\infty\) is the corresponding Lagrange multiplier, and the Euler-Lagrange equation is
\[
-\Delta U_a+VU_a
=
\mu_a^\infty U_a+aU_a^3
\ \text{in }\mathbb R^2.
\]
Consider the quadratic form
\[
\mathfrak q_{a,U_a}(h,k)
=
\int_{\mathbb R^2}
\left[
\nabla h(x)\cdot\nabla k(x)
+
(V(x)-\mu_a^\infty)h(x)k(x)
-
3aU_a^2(x)h(x)k(x)
\right] \mathrm dx,
\]
and write
\[
\mathfrak q_{a,U_a}(h):=\mathfrak q_{a,U_a}(h,h).
\]
Recall that the tangent space to the mass
constraint at \(U_a\) is
\[
T_{U_a}
=
\left\{
h\in H_V^1(\mathbb R^2):
\langle U_a,h\rangle=0
\right\}.
\]
Since \(U_a\) is a constrained minimizer,
\[
\mathfrak q_{a,U_a}(h)\ge0,
\ \forall h\in T_{U_a}.
\]
Moreover, we define the two scaled norms
\[
\|h\|_{*,a}^2
:=
\|h\|_{L^2(\mathbb R^2)}^2
+
\delta_a^2\|\nabla h\|_{L^2(\mathbb R^2)}^2
\]
and
\[
\|h\|_{\#,a}^2
:=
\|h\|_{*,a}^2
+
\delta_a^2\int_{\mathbb R^2}V(x)h^2(x)\mathrm dx.
\]
For every fixed \(a<a^*\), the \(\#\)-norm is equivalent to
\(\|\cdot\|_{H_V^1(\mathbb R^2)}\).

Set
\[
e_a:=\frac{p_a}{|p_a|},
\
J:=
\begin{pmatrix}
0&-1\\
1&0
\end{pmatrix},
\]
so that \(J\) is rotation by \(\pi/2\). Define
\[
Z_{0,a}
:=
-\delta_a(Jx)\cdot\nabla U_a,
\
Z_{1,a}
:=
-\delta_a e_a\cdot\nabla U_a,
\
Z_{2,a}
:=
U_a+(x-p_a)\cdot\nabla U_a.
\]

Without loss of generality, assume that \(e_a=p_a/|p_a|=e_1=(1,0)\). The blow-up
representatives of \(Z_{0,a}\), \(Z_{1,a}\), and \(Z_{2,a}\) are
\[
\delta_a Z_{0,a}(p_a+\delta_a\zeta)
=
-\bigl(Jp_a+\delta_aJ\zeta\bigr)\cdot\nabla W_a(\zeta),
\]
\[
\delta_a Z_{1,a}(p_a+\delta_a\zeta)
=
-e_1\cdot\nabla W_a(\zeta),\\
\]
and
\[
\delta_a Z_{2,a}(p_a+\delta_a\zeta)
=
W_a(\zeta)+\zeta\cdot\nabla W_a(\zeta).
\]
For any sequence $\{a_k\}$ satisfying $a_k\nearrow a^*$,
up to a subsequence,
still denoted by \(\{a_k\}\), we may assume that all the
convergences established in Section \ref{sec:Blow-up} hold along this subsequence.
To simplify the notation, throughout this section we omit the
index \(k\) and write
\[
a,\  U_a,\  p_a,\  \delta_a,\  \mu_a^\infty,\ldots
\]
instead of
\[
a_k,\  U_{a_k},\  p_{a_k},\
\delta_{a_k},\  \mu_{a_k}^\infty,\ldots .
\]
Accordingly, all limits as \(a\nearrow a^*\) below are understood
along this fixed subsequence.
In this sense, we shall also use the regularity estimates in Lemma~\ref{lem:whole}. If
\[
W_a(\zeta):=\delta_aU_a(p_a+\delta_a\zeta),\ \zeta\in\mathbb R^2,
\]
then as \(a\nearrow a^*\),
\[
W_a\to q
\ \text{ in }H^2_{\rm loc}(\mathbb R^2),
\]
where $q=\frac{Q}{\sqrt{a^*}}$.
The pointwise estimate \eqref{eqn:W-A-decay} and the uniform polynomially
weighted \(L^2(\mathbb R^2)\)-bound \eqref{eqn:weighted-L-2-bound} for derivatives through order three
recorded in Lemma~\ref{lem:whole} will be used when derivatives of the
translation, scaling, or rotation parameters are taken below.
Moreover,
 Lemma~\ref{lem:whole} implies that all blow-up
representatives of \(Z_{0,a}\), \(Z_{1,a}\), and \(Z_{2,a}\)
converge  in \(H^1(\mathbb R^2)\) to
\[
-\partial_2q,\ -\partial_1q,\ \Lambda q,
\]
respectively. Here $\Lambda q=q+x\cdot \nabla q$.
By the exponential decay in Lemma~\ref{lem:whole},
\[
Z_{j,a}\in H_V^1(\mathbb R^2),
\ j=0,1,2.
\]
The corresponding rotation, translation, and scaling curves preserve the
\(L^2(\mathbb R^2)\)-norm. Hence \(Z_{j,a}\in T_{U_a}\). Since \(V\) is radial,
rotations are exact symmetries of \(E_a^\infty\), and therefore
\[
\mathfrak q_{a,U_a}(Z_{0,a},h)=0,
\  h\in T_{U_a}.
\]
The corresponding curves are given by Steps 1, 2, and 4 below.  Now, we state the uniform quadratic growth.
\begin{lemma}
\label{lem:uniformcoercive}
Along the subsequence fixed above, for \(a\) sufficiently close to $a^*$,
there exist \(c>0\) and \(\rho_0>0\), independent of
$a$ and \(U_a\in\mathcal G_a\), such that the following hold.
\begin{itemize}
\item[$(i)$]\emph{(Coercivity modulo rotation, translation, and scaling directions).}
If \(h\in T_{U_a}\) and
\[
\langle h,Z_{0,a}\rangle
=
\langle h,Z_{1,a}\rangle
=
\langle h,Z_{2,a}\rangle
=0,
\]
then
\begin{equation}
\label{eq:sec23-074}
\mathfrak q_{a,U_a}(h)
\ge
c\delta_a^{-2}\|h\|_{\#,a}^2.
\end{equation}

\smallskip
\noindent
\item[$(ii)$]\emph{(Quadratic growth near \(\mathcal{G}_a\)).}
If
\[
v\in H_V^1(\mathbb R^2),
\
\|v\|_{L^2(\mathbb R^2)}=1,
\
\operatorname{dist}_{*,a}(v,\mathcal G_a)<\rho_0,
\]
then
\begin{equation}
\label{eq:sec23-075}
E_a^\infty(v)-e_\infty(a)
\ge
c\delta_a^2
\operatorname{dist}_{*,a}(v,\mathcal G_a)^2,
\end{equation}
where
\[
\operatorname{dist}_{*,a}(v,\mathcal G_a)
:=
\inf_{\widetilde U_a\in\mathcal G_a}
\|v-\widetilde U_a\|_{*,a}.
\]
\end{itemize}
\end{lemma}
\begin{proof}
We first outline the strategy of the proof and then carry it out in detail.

\medskip
\noindent\textbf{Outline of the proof.}
The proof proceeds in eight steps. Steps~1-2 compute the second variations
of the energy in the radial-translation and scaling directions, which are the
two almost-degenerate directions of \(\mathfrak{q}_{a,U_a}\) besides the
exact rotational zero direction. Step~3 proves the coercivity estimate
\[
\mathfrak{q}_{a,U_a}(h)\geq c\delta_a^{-2}\|h\|_{*,a}^2
\]
under the four orthogonal conditions, which is then
upgraded to the weighted \(\#\)-norm estimate \eqref{eq:sec23-074}. Steps~4-5 remove the exact rotational degeneracy by an exact
modulation and perform a two-dimensional Lyapunov-Schmidt reduction in the
directions \(\widehat{Z}_{1,a}\) and \(\widehat{Z}_{2,a}\). Step~6 proves that
the reduced energy is uniformly positive at the scale
\(\delta_a^2|z|^2\). Finally, Steps~7-8 combine the reduced-space and
complementary-space estimates to obtain the quadratic growth estimate \eqref{eq:sec23-075}, first in the \(\#\)-norm
and then in the \(\ast\)-norm.\\
\medskip
\noindent\textbf{Detailed proof.}
Throughout the proof, all constants below are uniform for \(a\) sufficiently close to $a^*$ and
\(U_a\in\mathcal G_a\).
For every integer \(k\ge1\), the symbol \(D^k\) denotes the \(k\)-th Fr\'echet derivative of the indicated map.
Moreover, all limits are understood as \(a\nearrow a^*\), unless otherwise stated.
\\
\medskip
\noindent
\textbf{Step 1. Second variation in the radial translation direction.}

Set
$p_a=\rho_a e_a,$
$\rho_a:=|p_a|
$
and
$$
U_{a,b}(x)
:=
U_a(x-b\delta_ae_a).
$$
This curve lies in \(H_V^1(\mathbb R^2)\) and preserves the
\(L^2(\mathbb R^2)\)-norm.
Moreover,
\[
\frac{d}{db}U_{a,b}\Big|_{b=0}
=
Z_{1,a}
\]
and
\[
\begin{aligned}
E_a^\infty(U_{a,b})
&=
\int_{\mathbb R^2}
\left(|\nabla U_a(y)|^2-\frac a2U_a^4(y)\right)\mathrm dy
+
\int_{\mathbb R^2}
V(y+b\delta_ae_a)U_a^2(y)\mathrm dy.
\end{aligned}
\]

For a unit vector \(e\in\mathbb S^1\), we use
\(\partial_{ee}V(x):=e^T D^2V(x)e\). For \(x\neq0\),
\[
D^2V(x)
=
2\widehat x\otimes\widehat x
+
2\frac{|x|-1}{|x|}
\left(
I-\widehat x\otimes\widehat x
\right),
\]
where \(\widehat x:=x/|x|\), \(\otimes\) denotes the tensor product, and \(I\) is the \(2\times2\) identity matrix.
Moreover,
\begin{equation*}
\pa_{e_ae_a}V(x)
=
2(\widehat x\cdot e_a)^2
+
2\frac{|x|-1}{|x|}
\left[
1-(\widehat x\cdot e_a)^2
\right]=2-\frac{2}{|x|}
\bigl[1-(\widehat x\cdot e_a)^2\bigr].
\end{equation*}
In particular,
\begin{equation}\label{eqn:V-upper}
\bigl|\partial_{e_ae_a}V(x)\bigr|
\leq
C\left(1+\frac1{|x|}\right),
\  x\neq0.
\end{equation}
Notice that the singularity at the origin is locally integrable in
dimension two and \(V\in W^{2,1}_{\rm loc}(\mathbb R^2)\).
For each fixed $a>0$, it is easy to see that \(U_a\) is locally bounded and decays exponentially at infinity. Hence
the map
\[
b\longmapsto
\int_{\mathbb R^2}V(y+b\delta_ae_a)U_a^2(y)\,\mathrm dy
\]
is twice differentiable near \(b=0\), with the derivatives obtained by
differentiating \(V\) under the integral.

Now, we claim that for every fixed \(R>0\),
\begin{equation}\label{eqn:partial-e-a-e-a-V}
\partial_{e_ae_a}V(p_a+\delta_a \zeta)
\to 2
\ \text{uniformly for }|\zeta|\leq R.
\end{equation}
Indeed, from Lemma~\ref{lem:whole}, we know that
\begin{equation}
\label{eq:sec23-081}
\frac{|p_a|-1}{\delta_a}\to0,
\end{equation}
which means \(\rho_a=|p_a|\to1\).
Decompose
\[
\zeta=\zeta_\parallel e_a+\zeta_\perp,
\]
where
\[
\zeta_\parallel:=\zeta\cdot e_a,
\
\zeta_\perp:=\zeta-\zeta_\parallel e_a
\]
and $\zeta_\perp\cdot e_a=0.$
Then
$
p_a+\delta_a \zeta
=
(\rho_a+\delta_a \zeta_\parallel)e_a+\delta_a \zeta_\perp,
$
and for every fixed \(R>0\) and all \(a\) sufficiently
close to \(a^*\),
\[
|p_a+\delta_a \zeta|\geq\frac12
,\ |\zeta|\leq R.
\]
Furthermore,
\[
1-
\left(
\frac{p_a+\delta_a \zeta}{|p_a+\delta_a \zeta|}
\cdot e_a
\right)^2
=
\frac{\delta_a^2|\zeta_\perp|^2}
{|p_a+\delta_a \zeta|^2}.
\]
It follows that
\[
\begin{aligned}
&
\left|
\partial_{e_ae_a}V(p_a+\delta_a \zeta)-2
\right| =
\frac{2\delta_a^2|\zeta_\perp|^2}
{|p_a+\delta_a \zeta|^3}
\leq C_R\delta_a^2,
\  |\zeta|\leq R.
\end{aligned}
\]
Therefore,
$
\sup_{|\zeta|\leq R}
\left|
\partial_{e_ae_a}V(p_a+\delta_a \zeta)-2
\right|
\to 0
$
and the claim holds.

Next, we justify  that
\begin{equation}\label{eq:radial-hessian-limit}
\int_{\mathbb R^2}
\partial_{e_ae_a}V(x)\,U_a^2(x)\mathrm dx
\to  2.
\end{equation}
 Recall from Lemma~\ref{lem:whole} that
\[
W_a(\zeta)=\delta_a U_a(p_a+\delta_a \zeta)
\to  q(\zeta)
\ \text{ in }H^1(\mathbb R^2),
\]
and $\int_{\mathbb R^2}q^2(\zeta)\mathrm d\zeta=1.$
For some constants \(C>0\) and \(\sigma>0\), independent of
\(a\) sufficiently close to \(a^*\),
\[
W_a(\zeta)\leq C e^{-\sigma |\zeta|}
\ \text{for }|\zeta|\ \text{sufficiently large}.
\]
We first consider the concentration region $B_{R\delta_a}(p_a)$.
Fix \(R>0\).  Since \(\rho_a=|p_a|\to1\), using the change of variables
$x=p_a+\delta_a \zeta,$
we obtain
\[
\begin{aligned}
&\int_{B_{R\delta_a}(p_a)}
\partial_{e_ae_a}V(x)U_a^2(x)\mathrm dx=
\int_{B_R(0)}
\partial_{e_ae_a}V(p_a+\delta_a \zeta)
W_a^2(\zeta)\mathrm d\zeta.
\end{aligned}
\]
Since \(W_a\to q\)  in \(L^2(\mathbb R^2)\), then by \eqref{eqn:partial-e-a-e-a-V},
\begin{equation}
\int_{B_{R\delta_a}(p_a)}
\partial_{e_ae_a}V(x)U_a^2(x)\mathrm dx
\to
2\int_{B_R(0)}q^2(\zeta)\mathrm d\zeta
\label{eq:local-hessian-contribution}
\end{equation}
for every fixed \(R>0\).
For $\mathbb R^2\setminus B_{R\delta_a}(p_a)$,
we claim that
\begin{equation}
\lim_{R\to\infty}
\limsup_{a\nearrow a^*}
\int_{\mathbb R^2\setminus B_{R\delta_a}(p_a)}
\bigl|\partial_{e_ae_a}V(x)\bigr|
U_a^2(x)\mathrm dx
=0.
\label{eq:hessian-far-field contribution}
\end{equation}
Because of the possible singularity of \(D^2V\) at the origin, we
split the integral into the regions \(|x|<1/4\) and \(|x|\geq1/4\).
Since \(|p_a|\to1\), for \(a\) sufficiently close to \(a^*\), one has
$|p_a|\geq\frac34.$
Thus, if \(|x|<1/4\),
\[
|x-p_a|
\geq |p_a|-|x|
\geq\frac12.
\]
The uniform exponential decay of \(W_a\) in Lemma \ref{lem:whole} therefore yields
\[
\begin{aligned}
U_a(x)
=
\delta_a^{-1}
W_a\left(\frac{x-p_a}{\delta_a}\right)\leq
C\delta_a^{-1}
\exp\left(-\frac{\sigma}{2\delta_a}\right)
\ \text{for }|x|<\frac14.
\end{aligned}
\]
Using \eqref{eqn:V-upper},
we obtain
\begin{equation}\label{eq:sec23-083}
\begin{aligned}
\int_{|x|<1/4}
|\partial_{e_ae_a}V(x)|U_a^2(x)\mathrm dx
\leq
C\delta_a^{-2}e^{-\sigma/\delta_a}
\int_{|x|<1/4}
\left(1+\frac1{|x|}\right)\mathrm dx=o(1).
\end{aligned}
\end{equation}
On the other hand, for \(|x|\geq1/4\),
$|\partial_{e_ae_a}V(x)|\leq C.$
Hence
\[
\begin{aligned}
\int_{\{|x-p_a|>R\delta_a\}\cap\{|x|\geq1/4\}}
|\partial_{e_ae_a}V(x)|U_a^2(x)\mathrm dx
\leq
C\int_{|x-p_a|>R\delta_a}U_a^2(x)\mathrm dx
\ =
C\int_{|\zeta|>R}W_a^2(\zeta)\mathrm d\zeta.
\end{aligned}
\]
Since \(W_a\to q\)  in \(L^2(\mathbb R^2)\),
\[
\limsup_{a\nearrow a^*}
\int_{|\zeta|>R}W_a^2(\zeta)\mathrm d\zeta
\leq
\int_{|\zeta|>R}q^2(\zeta)\mathrm d\zeta.
\]
Letting \(R\to\infty\), we obtain
\eqref{eq:hessian-far-field contribution}.
Combining \eqref{eq:local-hessian-contribution} and
\eqref{eq:hessian-far-field contribution},
we arrive at
\[
\int_{\mathbb R^2}
\partial_{e_ae_a}V(x)U_a^2(x)\mathrm dx
\to
2\int_{\mathbb R^2}q^2(\zeta)\mathrm d\zeta
=2,
\]
which proves \eqref{eq:radial-hessian-limit}.

Recalling that $\frac{d}{db}U_{a,b}\Big|_{b=0}=Z_{1,a},$
the constrained second-variation formula yields
\[
\frac{d^2}{db^2}E_a^\infty(U_{a,b})\Big|_{b=0}
=
2\mathfrak q_{a,U_a}(Z_{1,a})=\delta_a^2
\int_{\mathbb R^2}
\partial_{e_ae_a}V(x)U_a^2(x)\mathrm dx.
\]
Consequently,
\begin{equation*}
\mathfrak q_{a,U_a}(Z_{1,a})
=
\delta_a^2+o(\delta_a^2).
\end{equation*}

\medskip
\noindent
\textbf{Step 2. Second variation in the scaling direction.}

Consider the mass-preserving scaling
\[
U_{a,\tau}(x)
:=
e^\tau
U_a\bigl(p_a+e^\tau(x-p_a)\bigr).
\]
Then
\[
\frac{d}{d\tau}U_{a,\tau}\Big|_{\tau=0}
=
Z_{2,a}.
\]
Setting
\begin{equation}
\label{eq:sec23-087}
A_a
:=
\int_{\mathbb R^2}|\nabla U_a(x)|^2\mathrm dx
-
\frac a2\int_{\mathbb R^2}U_a^4(x)\mathrm dx,
\end{equation}
we have
\begin{equation}
\label{eq:sec23-088}
E_a^\infty(U_{a,\tau})
=
e^{2\tau}A_a
+
\int_{\mathbb R^2}
V\bigl(p_a+e^{-\tau}(y-p_a)\bigr)U_a^2(y)\mathrm dy.
\end{equation}
As in the proof of \eqref{eqn:P-a-est}, we can infer from Lemma~\ref{lem:whole} that
\[
\int_{\mathbb R^2}V(x)U_a^2(x)\mathrm dx
=\delta_a^2\int_{\mathbb R^2}(e_a\cdot \zeta)^2q^2(\zeta)\mathrm d\zeta+o(\delta_a^2).
\]
Since $q$ is radial, after a rotation we may assume $e_a=e_1=(1,0)$.
Hence
\[
\delta_a^{-2}\int_{\mathbb R^2}V(x)U_a^2(x)\mathrm dx
\to
\int_{\mathbb R^2}\zeta_1^2q^2(\zeta)\mathrm d\zeta
=\frac{\lambda_0^4}{a^*}
\]
and
\begin{equation}
\label{eq:sec23-089}
A_a=e_{\infty}(a)-\int_{\mathbb R^2}V(x)U_a^2(x)\mathrm dx
=
\gamma\,\delta_a^2
+
o(\delta_a^2),
\
\gamma
:=
\int_{\mathbb R^2}\zeta_1^2q^2(\zeta)\mathrm d\zeta
=
\frac{\lambda_0^4}{a^*}.
\end{equation}
If \(\mathcal P_{a,U_a}(\tau)\) denotes the potential term in \eqref{eq:sec23-088}, then
\[
\mathcal P_{a,U_a}''(0)
=
\int_{\mathbb R^2}
\left[
(y-p_a)^TD^2V(y)(y-p_a)
+
\nabla V(y)\cdot(y-p_a)
\right]U_a^2(y)\mathrm dy.
\]
Let \(y=p_a+\delta_a\zeta\). Using Lemma~\ref{lem:whole} and arguing as in Step 1, we obtain
\[
\delta_a^{-2}\mathcal P_{a,U_a}''(0)
\to
4\int_{\mathbb R^2} \zeta_1^2q^2(\zeta)\mathrm d\zeta
=
4\gamma.
\]
 Hence
\[
\frac{d^2}{d\tau^2}
E_a^\infty(U_{a,\tau})\Big|_{\tau=0}
=
8\gamma\,\delta_a^2
+
o(\delta_a^2),
\]
and therefore
\begin{equation*}
\mathfrak q_{a,U_a}(Z_{2,a})
=
4\gamma\,\delta_a^2
+
o(\delta_a^2).
\end{equation*}

\medskip
\noindent
\textbf{Step 3. Coercivity modulo rotation, translation, and scaling directions.}

We first prove the unweighted estimate
\begin{equation}
\label{eq:sec23-093}
\mathfrak q_{a,U_a}(h)
\ge
c\delta_a^{-2}\|h\|_{*,a}^2
\end{equation}
under the four orthogonality conditions
\begin{equation}
\label{eq:sec23-094}
h\perp U_a,\
h\perp Z_{0,a},\
h\perp Z_{1,a},\
h\perp Z_{2,a}.
\end{equation}
Assume that \eqref{eq:sec23-093} is false. Then there are
\(a_k\nearrow a^*\), \(U_k\in\mathcal G_{a_k}\), and \(h_k\) satisfying the
four orthogonality conditions in \eqref{eq:sec23-094}. Write \(\delta_k:=\delta_{a_k}\), let \(p_k\) be a
global maximum point of \(U_k\), and set \(\mu_k:=\mu_{a_k}^\infty\). Then as $k\to\infty$,
\begin{equation}
\label{eq:sec23-095}
\|h_k\|_{*,a_k}=1,
\
\delta_k^2\mathfrak q_{a_k,U_k}(h_k)\to0.
\end{equation}
After rotation, we take
\[
p_k=|p_k|e_1.
\]
Define
\begin{equation*}
\eta_k(\zeta)
=
\delta_kh_k(p_k+\delta_k\zeta),
\
W_k(\zeta)
=
\delta_kU_k(p_k+\delta_k\zeta),
\end{equation*}
then $\|\eta_k\|_{H^1(\mathbb R^2)}=1.$
Up to a subsequence, $\eta_k\rightharpoonup\eta_0
\ \text{in }H^1(\mathbb R^2)$ as $k\to\infty$.
The orthogonality conditions and the blow-up convergence established in
Lemma~\ref{lem:whole} imply
\begin{equation}
\label{eq:sec23-099}
\langle\eta_0,q\rangle
=
\langle\eta_0,\partial_1q\rangle
=
\langle\eta_0,\partial_2q\rangle
=
\langle\eta_0,\Lambda q\rangle
=0.
\end{equation}
Moreover,
\[
\begin{aligned}
\delta_k^2\mathfrak q_{a_k,U_k}(h_k)
&=
\int_{\mathbb R^2}
\left(
|\nabla\eta_k(\zeta)|^2
-\mu_k\delta_k^2\eta_k^2(\zeta)
-3a_kW_k^2(\zeta)\eta_k^2(\zeta)
\right)\mathrm d\zeta+
\delta_k^2
\int_{\mathbb R^2}
V(p_k+\delta_k \zeta)\eta_k^2(\zeta)\mathrm d\zeta.
\end{aligned}
\]
Since \(\{\eta_k\}\) is bounded in \(H^1(\mathbb R^2)\), it is bounded in
\(L^4(\mathbb R^2)\), and consequently
\[
\int_{\mathbb R^2}
\bigl(a_kW_k^2(\zeta)-Q^2(\zeta)\bigr)\eta_k^2(\zeta)\,\mathrm d\zeta
\to0
\]
and
\[
Q\eta_k\to Q\eta_0
\
\text{ in }L^2(\mathbb R^2)
\]
as $k\to\infty.$
Hence
\[
\int_{\mathbb R^2}a_kW_k^2(\zeta)\eta_k^2(\zeta)\,\mathrm d\zeta
\to
\int_{\mathbb R^2}Q^2(\zeta)\eta_0^2(\zeta)\,\mathrm d\zeta.
\]
Together with \(-\mu_k\delta_k^2\to1\) as $k\to\infty$ and
\(\|\eta_k\|_{H^1(\mathbb R^2)}=1\), this gives
\begin{equation}
\label{eq:sec23-101}
\begin{aligned}
\liminf_{k\to\infty}
\delta_k^2\mathfrak q_{a_k,U_k}(h_k)
&\ge
1-3\int_{\mathbb R^2}Q^2(\zeta)\eta_0^2(\zeta)\mathrm d\zeta
\\
&=
\mathcal B(\eta_0,\eta_0)
+
1-\|\eta_0\|_{H^1(\mathbb R^2)}^2.
\end{aligned}
\end{equation}
By Lemma~\ref{lem:Qcoercive} and \eqref{eq:sec23-099},
\[
\mathcal B(\eta_0,\eta_0)
\ge
c_0\|\eta_0\|_{H^1(\mathbb R^2)}^2.
\]
Thus the right-hand side of \eqref{eq:sec23-101} is bounded below by
\[
\min\{1,c_0\}>0,
\]
contradicting \eqref{eq:sec23-095}. Namely, \eqref{eq:sec23-093} holds.

We now upgrade \eqref{eq:sec23-093} to the weighted estimate. Since $-\mu_a^\infty\delta_a^2\to1,$
we have \(\mu_a^\infty<0\) for \(a\) close to \(a^*\). Therefore
\begin{equation}
\label{eq:sec23-102}
\begin{aligned}
\int_{\mathbb R^2}V(x)h^2(x)\mathrm dx
&=
\mathfrak q_{a,U_a}(h)
-\|\nabla h\|_{L^2(\mathbb R^2)}^2
+\mu_a^\infty\|h\|_{L^2(\mathbb R^2)}^2
+3a\int_{\mathbb R^2}U_a^2(x)h^2(x)\mathrm dx
\\
&\le
\mathfrak q_{a,U_a}(h)
+
3a\int_{\mathbb R^2}U_a^2(x)h^2(x)\mathrm dx.
\end{aligned}
\end{equation}
By the Gagliardo-Nirenberg inequality,
\begin{equation}
\label{eq:sec23-103}
\|h\|_{L^4(\mathbb R^2)}^4
\le
C\|h\|_{L^2(\mathbb R^2)}^2\|\nabla h\|_{L^2(\mathbb R^2)}^2
\le
C\delta_a^{-2}\|h\|_{*,a}^4,
\end{equation}
and Lemma~\ref{lem:whole} implies $\|U_a\|_{L^4(\mathbb R^2)}
\le C\delta_a^{-1/2}.$
Hence
\begin{equation}
\label{eq:sec23-104}
\delta_a^2\int_{\mathbb R^2}U_a^2(x)h^2(x)\mathrm dx
\le
C\|h\|_{*,a}^2.
\end{equation}
Combining \eqref{eq:sec23-093}, \eqref{eq:sec23-102} and \eqref{eq:sec23-104},
\[
\delta_a^2\int_{\mathbb R^2}V(x)h^2(x)\mathrm dx
\le
C\delta_a^2\mathfrak q_{a,U_a}(h).
\]
Together with the definition of $\|h\|_{\#,a}$, we obtain
\[
\|h\|_{\#,a}^2
\le
C\delta_a^2\mathfrak q_{a,U_a}(h),
\]
which is exactly \eqref{eq:sec23-074}. This proves part $(i)$.

\medskip
\noindent
\textbf{Step 4. Exact rotation modulation.}

For \(s\in\mathbb R\), set
\[
U_{a,s}(x)
:=
U_a(R_{-\delta_a s}x),
\]
where
\[
R_{\vartheta}
:=
\begin{pmatrix}
\cos\vartheta & -\sin\vartheta\\
\sin\vartheta & \cos\vartheta
\end{pmatrix}
=e^{\vartheta J},
\
J:=
\begin{pmatrix}
0&-1\\
1&0
\end{pmatrix},
\]
so that \(R_{-\vartheta}=R_\vartheta^{-1}\).
Thus the rotation angle is \(\vartheta=\delta_a s\), and
\[
Z_{0,a,s}
:=
\partial_s U_{a,s}
=
-\delta_a(Jx)\cdot\nabla U_{a,s},
\
Z_{0,a,0}=Z_{0,a}.
\]
Since \(V\) is radial,
$U_{a,s}\in\mathcal G_a, \ \text{and}\
E_a^\infty(U_{a,s})=e_\infty(a).
$
Then
$$
\langle Z_{0,a},U_a\rangle=0
$$
and
 \[
\begin{aligned}
0
&=
\left.
\frac{d}{ds}
\left(
DE_a^\infty(U_{a,s})[h]
-
2\mu_a^\infty\langle U_{a,s},h\rangle
\right)
\right|_{s=0}
\\
&=
2\mathfrak q_{a,U_a}(Z_{0,a},h), \ h\in T_{U_a}.
\end{aligned}
\]

For any \(\widetilde U_a\in\mathcal G_a\)
and \(\xi\in T_{\widetilde U_a}\),
 define the mass-chart
\[
\mathcal C_{\widetilde U_a}(\xi)
:=
\frac{\widetilde U_a+\xi}
{\|\widetilde U_a+\xi\|_{L^2(\mathbb R^2)}}.
\]
For \(s\) near zero and  \(v\) sufficiently close to \(U_a\) in the sense of \(*\)-norm, define
\[
m(s,v):=\langle v,U_{a,s}\rangle,
\
\Xi(s,v):=\frac{v}{m(s,v)}-U_{a,s}.
\]
Then $\langle\Xi(s,v),U_{a,s}\rangle=0,$
and, as long as \(m(s,v)>0\),
\[
v
=
\mathcal C_{U_{a,s}}\bigl(\Xi(s,v)\bigr)
=
\frac{U_{a,s}+\Xi(s,v)}
{\|U_{a,s}+\Xi(s,v)\|_{L^2(\mathbb R^2)}}.
\]
Set
\[
\mathcal F_a(s,v)
:=
\langle\Xi(s,v),Z_{0,a,s}\rangle.
\]
At \((s,v)=(0,U_a)\), we can infer from \(\langle U_a,Z_{0,a}\rangle=0\) that
\[
\partial_s\Xi(0,U_a)=-Z_{0,a}.
\]
Therefore
\begin{equation}
\label{eq:rotation-modulation-derivative}
\partial_s\mathcal F_a(0,U_a)
=
-\|Z_{0,a}\|_{L^2(\mathbb R^2)}^2.
\end{equation}
We next verify that this derivative is uniformly separated from zero.
Recall that
\[
\delta_a Z_{0,a}(p_a+\delta_a\zeta)
=
-\bigl(Jp_a+\delta_aJ\zeta\bigr)\cdot\nabla W_a(\zeta)
\]
and $\delta_a Z_{0,a}(p_a+\delta_a\zeta)$ converges to $-\partial_2q$ in $L^2(\mathbb R^2)$.
 Hence
\[
\|Z_{0,a}\|_{L^2(\mathbb R^2)}^2
\to
\int_{\mathbb R^2}|\partial_2q(\zeta)|^2\mathrm d\zeta
=
\|\partial_2q\|_{L^2(\mathbb R^2)}^2>0.
\]
In particular, the derivative in
\eqref{eq:rotation-modulation-derivative} is uniformly invertible.

We next prove the uniform \(C^2\)-regularity needed to apply the implicit function theorem to \(\mathcal F_a(s,v)=0\).
First, we control the parameter derivatives of \(U_{a,s}\). For \(j=1,2,3\),
the derivative \(\partial_s^jU_{a,s}\) is a finite linear combination of terms
obtained by applying \(\delta_a(Jx)\cdot\nabla\) at most \(j\) times to
\(U_{a,s}\). By the change of variables \(x=p_a+\delta_a\zeta\), these
terms become polynomially weighted derivatives of \(W_a\) of order at most
\(j\). Hence the weighted derivative estimates in Lemma~\ref{lem:whole} give
\[
\sup_{|s|\le s_0}
\|\partial_s^jU_{a,s}\|_{L^2(\mathbb R^2)}
\le C,
\
j=1,2,3,
\]
for a fixed \(s_0>0\), uniformly in \(a\) and \(U_a\). In particular,
\(s\mapsto U_{a,s}\) is uniformly \(C^3\) as an
\(L^2(\mathbb R^2)\)-valued map.
Second, we transfer this regularity to the modulation map $\mathcal F_a$. Since
\(\mathcal F_a\) is formed only from \(L^2(\mathbb R^2)\)-pairings of \(v\),
\(U_{a,s}\), and \(Z_{0,a,s}=\partial_sU_{a,s}\), the preceding bound implies
that \((s,v)\mapsto\mathcal F_a(s,v)\) is uniformly \(C^2\) on a fixed
neighborhood of \((0,U_a)\). Taking this neighborhood small enough,
\(m(s,v)\) is also uniformly bounded away from zero.
 By
\eqref{eq:rotation-modulation-derivative}, the derivative
\(\partial_s\mathcal F_a(0,U_a)\) is uniformly separated from zero. Therefore
the implicit function theorem (see, e.g., \cite{AmbrosettiProdi1993}) provides constants \(r_{\rm rot}>0\) and
\(C>0\), independent of \(a\) and \(U_a\), and a unique \(C^2\) function
\(s=s(v)\) when \(\|v-U_a\|_{*,a}<r_{\rm rot}\), such that
\begin{equation}
\label{eq:rotation-slice}
\left\langle
\Xi(s(v),v),Z_{0,a,s(v)}
\right\rangle=0.
\end{equation}
Moreover, \(|s(v)|\le C\|v-U_a\|_{*,a}\). Since the above modulation only uses
\(L^2(\mathbb R^2)\)-pairings and rotations preserve both the \(*\)- and
\(\#\)-norms, the same construction is valid in either neighborhood.

\medskip
\noindent
\textbf{Step 5. Lyapunov-Schmidt reduction in \(H_V^1(\mathbb R^2)\).}

For fixed
\(U_a\in\mathcal G_a\), set
\[
\mathcal X_{a,U_a}
:=
\left\{
h\in T_{U_a}:
\langle h,Z_{0,a}\rangle=0
\right\}.
\]
Define
\[
\widehat Z_{j,a}
:=
Z_{j,a}
-
\frac{\langle Z_{j,a},Z_{0,a}\rangle}
{\|Z_{0,a}\|_{L^2(\mathbb R^2)}^2}
Z_{0,a},
\
j=1,2,
\]
and
\[
\mathcal Z_{a,U_a}
=
\operatorname{span}
\{\widehat Z_{1,a},\widehat Z_{2,a}\},\
\mathcal H_{a,U_a}
:=
\left\{
h\in\mathcal X_{a,U_a}:
\langle h,\widehat Z_{1,a}\rangle
=
\langle h,\widehat Z_{2,a}\rangle
=0
\right\}.
\]
Note that, under the blow-up scaling, the Gram matrix of
\(\{Z_{0,a},Z_{1,a},Z_{2,a}\}\) converges to the Gram matrix of
\(\{-\partial_2q,-\partial_1q,\Lambda q\}\), which is positive definite.
Consequently,
\[
\mathcal X_{a,U_a}
=
\mathcal Z_{a,U_a}\oplus\mathcal H_{a,U_a}
\]
with uniformly bounded \(L^2(\mathbb R^2)\)-orthogonal projections
onto
\(\mathcal Z_{a,U_a}\). Since the almost-degenerate directions $Z_{j,a}$ ($j=1,2$) have uniformly bounded
\(\#\)-norms and the projection coefficients are controlled by the
\(L^2(\mathbb R^2)\)-norm, these projections are also uniformly bounded in the
\(*\)- and \(\#\)-norms.
In particular, if
\[
z:=b\widehat Z_{1,a}+\tau\widehat Z_{2,a},
\]
then there exist constants $c_1,c_2>0$ such that
\begin{equation}
\label{eq:sec23-131}
c_1(b^2+\tau^2)
\le
\|z\|_{*,a}^2
\le
\|z\|_{\#,a}^2
\le
c_2(b^2+\tau^2).
\end{equation}
Indeed, let $z=b\widehat Z_{1,a}+\tau\widehat Z_{2,a}\in\mathcal Z_{a,U_a}$
and define
\[
\eta_z(y):=\delta_a z(p_a+\delta_a y),
\  y\in\mathbb R^2.
\]
The scaled versions of \(\widehat Z_{1,a}\) and \(\widehat Z_{2,a}\) converge
 in \(H^1(\mathbb R^2)\) to \(-\partial_1q\) and \(\Lambda q\),
respectively. Since these two limiting functions are linearly independent (in
fact \(L^2(\mathbb R^2)\)-orthogonal by parity), there exist constants
\(c_1,c_2>0\), independent of \(a\) and \(U_a\), such that
\[
c_1(b^2+\tau^2)
\le
\|\eta_z\|_{H^1(\mathbb R^2)}^2
=
\|z\|_{*,a}^2
\le
c_2(b^2+\tau^2).
\]
Moreover, the additional term in the \(\#\)-norm satisfies
\[
\delta_a^2\int_{\mathbb R^2}V(x)z^2(x)\mathrm dx
=
\delta_a^2\int_{\mathbb R^2}V(p_a+\delta_a y)\eta_z^2(y)\mathrm dy.
\]
Since \(|p_a|-1=o(\delta_a)\), then
\[
\bigl||p_a+\delta_a y|-1\bigr|
\le
\bigl||p_a|-1\bigr|+\delta_a|y|
\le
C\delta_a(1+|y|),
\]
which signifies that
\[
V(p_a+\delta_a y)
\le
C\delta_a^2(1+|y|^2),
\  y\in\mathbb R^2.
\]
Meanwhile, the exponential estimate
\eqref{eqn:W-A-decay} implies
\[
\int_{\mathbb R^2}(1+|y|^2)\eta_z^2(y)\mathrm dy
\le C(b^2+\tau^2)\le
C\|\eta_z\|_{H^1(\mathbb R^2)}^2,
\]
with \(C\) independent of \(a\) and \(U_a\). Therefore
\[
\delta_a^2\int_{\mathbb R^2}V(x)z^2(x)\mathrm dx
\le
C\delta_a^4\|\eta_z\|_{H^1(\mathbb R^2)}^2
\]
and
\[
\|z\|_{\#,a}^2
=
\|z\|_{*,a}^2+O(\delta_a^4)\|z\|_{*,a}^2
=
\bigl(1+O(\delta_a^4)\bigr)\|z\|_{*,a}^2.
\]
Combining the above estimates, we conclude that \eqref{eq:sec23-131} holds.

Recall the mass-chart \(\mathcal C_{U_a}\) introduced in Step 4. For
\(\xi\in T_{U_a}\), define
\[
\mathscr E_{a,U_a}(\xi)
:=
E_a^\infty(\mathcal C_{U_a}(\xi))-e_\infty(a).
\]
Since the potential term is a continuous quadratic form on \(H_V^1(\mathbb R^2)\), while
\(H_V^1(\mathbb R^2)\hookrightarrow H^1(\mathbb R^2)\hookrightarrow L^p(\mathbb R^2)\) for every finite \(p\) in
dimension two, \(\mathscr E_{a,U_a}\) is \(C^\infty\) on a neighborhood of \(0\)
in \(H_V^1(\mathbb R^2)\). Moreover,
\[
D\mathscr E_{a,U_a}(0)=0
\]
and
\[
D^2\mathscr E_{a,U_a}(0)[h,k]
=
2\mathfrak q_{a,U_a}(h,k).
\]

In the following, for \(z\in\mathcal Z_{a,U_a}\), we seek
\(\Psi_{a,U_a}(z)\in\mathcal H_{a,U_a}\) satisfying
\begin{equation}
\label{eq:sec23-135}
D\mathscr E_{a,U_a}
\bigl(z+\Psi_{a,U_a}(z)\bigr)[\eta]
=
0,
\
\forall\,\eta\in\mathcal H_{a,U_a}.
\end{equation}
Set
\[
\mathcal N_{a,U_a}(z,h)[\eta]:
=
D\mathscr E_{a,U_a}(z+h)[\eta].
\]
Let \(D_h\) denote the Fr\'echet derivative with respect to the variable \(h\). Throughout Step 5, operator norms on
\(\mathcal H_{a,U_a}\) are taken with the \(\#\)-norm and the corresponding
dual norm on \(\mathcal H_{a,U_a}^*\).
First, we prove that the linear map
\[
A_{a,U_a}
:=
D_h\mathcal N_{a,U_a}(0,0)
:
\mathcal H_{a,U_a}\to \mathcal H_{a,U_a}^*
\]
is an isomorphism and
\begin{equation}
\label{eq:sec23-136}
\|A_{a,U_a}^{-1}\|
\le
C\delta_a^2.
\end{equation}
If \(h\in\mathcal H_{a,U_a}\), then
\(h\perp Z_{0,a}\) because \(\mathcal H_{a,U_a}\subset\mathcal X_{a,U_a}\),
and \(h\perp Z_{j,a}\) for \(j=1,2\) because
\(Z_{j,a}-\widehat Z_{j,a}\in\operatorname{span}\{Z_{0,a}\}\).
Thus all orthogonality conditions in Lemma~\ref{lem:uniformcoercive} $(i)$ hold. Using
 \eqref{eq:sec23-074}, we have
\[
\mathfrak q_{a,U_a}(h,h)\ge c\delta_a^{-2}\|h\|_{\#,a}^2.
\]
The bilinear form associated with \(A_{a,U_a}\) is
\[
\langle A_{a,U_a}h,k\rangle
=2\mathfrak q_{a,U_a}(h,k),
\ h,k\in\mathcal H_{a,U_a}.
\]
It is continuous and coercive on \(\mathcal H_{a,U_a}\), with coercivity
constant \(2c\delta_a^{-2}>0\).
Moreover,
\begin{equation*}
\label{eq:sec23-137}
|\mathfrak q_{a,U_a}(h,k)|
\le
C\delta_a^{-2}
\|h\|_{\#,a}\|k\|_{\#,a},
\end{equation*}
which follows directly from \(-\mu_a^\infty=O(\delta_a^{-2})\)
and \eqref{eq:sec23-104}.
Therefore, by the Lax-Milgram theorem, \(A_{a,U_a}:\mathcal H_{a,U_a}\to
\mathcal H_{a,U_a}^*\) is an isomorphism and
\[
\|A_{a,U_a}^{-1}\|
\le
\frac{1}{2c}\delta_a^2.
\]

Second, we prove the
uniform nonlinear estimate
\begin{equation}
\label{eq:sec23-138}
\left\|
D_h\mathcal N_{a,U_a}(z,h)
-
A_{a,U_a}
\right\|
\le
C\delta_a^{-2}
\left(
|z|+\|h\|_{\#,a}
\right)
\end{equation}
when \(|z|+\|h\|_{\#,a}\le r\), with \(r>0\) fixed and small.
Here \(z\in\mathcal Z_{a,U_a}\) is a finite-dimensional reduced vector, and we write
\[
z=b\widehat Z_{1,a}+\tau\widehat Z_{2,a}.
\]
The quantity $|z|:=\sqrt{b^2+\tau^2}$ denotes
the Euclidean norm of the coefficient vector.
By \eqref{eq:sec23-131}, this norm is equivalent to \(\|z\|_{*,a}\) and \(\|z\|_{\#,a}\)
on the finite-dimensional space \(\mathcal Z_{a,U_a}\), that is,
\[
c_1|z|^2\le\|z\|_{*,a}^2\le\|z\|_{\#,a}^2\le c_2|z|^2.
\]
We now verify \eqref{eq:sec23-138} in detail. Since
\(D_h\mathcal N_{a,U_a}(z,h)=D^2\mathscr E_{a,U_a}(z+h)\) and
\(A_{a,U_a}=D^2\mathscr E_{a,U_a}(0)\), we estimate the difference of the
Hessians. Write \(\xi=z+h\). Because \(\xi\in T_{U_a}\) and
\(\|U_a\|_{L^2(\mathbb R^2)}=1\),
\[
\|U_a+\xi\|_{L^2(\mathbb R^2)}^2
=
1+\|\xi\|_{L^2(\mathbb R^2)}^2.
\]
Let
\[
\mathcal A(w)
:=
\int_{\mathbb R^2}
\bigl(|\nabla w(x)|^2+V(x)w^2(x)\bigr)\mathrm dx,
\
\mathcal A(k,\eta)
:=
\int_{\mathbb R^2}
\bigl(\nabla k(x)\cdot\nabla\eta(x)+V(x)k(x)\eta(x)\bigr)\mathrm dx
\]
and
\[
\mathcal Q(w)
:=
\int_{\mathbb R^2}|w(x)|^4\mathrm dx,
\]
then the mass-chart gives the exact formula
\[
\mathscr E_{a,U_a}(\xi)
=
\frac{\mathcal A(U_a+\xi)}{1+\|\xi\|_{L^2(\mathbb R^2)}^2}
-
\frac{a}{2}
\frac{\mathcal Q(U_a+\xi)}{
\bigl(1+\|\xi\|_{L^2(\mathbb R^2)}^2\bigr)^2}
-e_\infty(a).
\]
The quadratic form \(\mathcal A\) satisfies
\[
|\mathcal A(k,\eta)|
\le
C\delta_a^{-2}\|k\|_{\#,a}\|\eta\|_{\#,a},
\
\mathcal A(U_a)\le C\delta_a^{-2},
\]
and hence the product rule, together with
\(D\|\xi\|_{L^2(\mathbb R^2)}^2=O(\|\xi\|_{L^2(\mathbb R^2)})\), shows that
all Hessian differences coming from the normalization factor ${1+\|\xi\|_{L^2(\mathbb R^2)}^2}$ and the
quadratic part are bounded by
\[
C\delta_a^{-2}\|\xi\|_{\#,a}
\|k\|_{\#,a}\|\eta\|_{\#,a}
\]
when \(\|\xi\|_{\#,a}\) is small.
It remains to estimate the quartic part. Set
\[
G(\xi)=\mathcal Q(U_a+\xi)
=
\int_{\mathbb R^2}|U_a(x)+\xi(x)|^4\mathrm dx.
\]
Then
\[
D^2G(\xi)[k,\eta]-D^2G(0)[k,\eta]
=
12\int_{\mathbb R^2}\bigl[2U_a(x)\xi(x)+\xi^2(x)\bigr]k(x)\eta(x)\mathrm dx.
\]
By H\"older's inequality and \eqref{eq:sec23-103},
\[
\|U_a\|_{L^4(\mathbb R^2)}\le C\delta_a^{-1/2},\
\|w\|_{L^4(\mathbb R^2)}\le C\delta_a^{-1/2}\|w\|_{*,a}
\ \text{for }w\in H^1(\mathbb R^2),
\]
so
\[
\left|\int_{\mathbb R^2} U_a(x)\xi(x)k(x)\eta(x)\mathrm dx\right|
\le
C\delta_a^{-2}\|\xi\|_{*,a}\|k\|_{*,a}\|\eta\|_{*,a},
\]
and
\[
\left|\int_{\mathbb R^2} \xi^2(x)k(x)\eta(x)\mathrm dx\right|
\le
C\delta_a^{-2}\|\xi\|_{*,a}^2\|k\|_{*,a}\|\eta\|_{*,a}.
\]
Since \(\|\xi\|_{\#,a}\le C(|z|+\|h\|_{\#,a})\) by
\eqref{eq:sec23-131}, then we conclude that
\[
\bigl|
(D_h\mathcal N_{a,U_a}(z,h)-A_{a,U_a})[k,\eta]
\bigr|
\le
C\delta_a^{-2}(|z|+\|h\|_{\#,a})
\|k\|_{\#,a}\|\eta\|_{\#,a}.
\]
Taking the supremum over \(\|k\|_{\#,a}=\|\eta\|_{\#,a}=1\), we get
\eqref{eq:sec23-138}.

Combining \eqref{eq:sec23-136} and \eqref{eq:sec23-138}, the product of the
inverse bound \(O(\delta_a^2)\) and the nonlinear Lipschitz bound
\(O(\delta_a^{-2}r)\) is \(O(r)\). Thus the smallness condition required by
the implicit function theorem (see, e.g., \cite{AmbrosettiProdi1993,Deimling1985}) can be chosen independently of \(a\) and
\(U_a\). We apply the implicit function theorem directly to
\[
\mathcal N_{a,U_a}:
\mathcal Z_{a,U_a}\times\mathcal H_{a,U_a}
\to
\mathcal H_{a,U_a}^*.
\]
There is a radius \(r_\#>0\), independent of \(a\) and \(U_a\), and a unique
\(C^2\)-map
\[
\Psi_{a,U_a}:
\mathbb B_{r_\#}^{\mathcal Z}(0)
\to
\mathcal H_{a,U_a},
\]
where for every \(r>0\),
\[
\mathbb B_{r}^{\mathcal Z}(0)
:=
\left\{
z=b\widehat Z_{1,a}+\tau\widehat Z_{2,a}\in\mathcal Z_{a,U_a}:
 b^2+\tau^2<r^2
\right\}.
\]
It satisfies \eqref{eq:sec23-135} and \(\Psi_{a,U_a}(0)=0\). Moreover, reducing
\(r_\#\) if necessary,
\begin{equation}\label{eqn:Psi-bound}
\|\Psi_{a,U_a}(z)\|_{\#,a}\le C|z|,
\
z\in\mathbb B_{r_\#}^{\mathcal Z}(0).
\end{equation}
Indeed, the implicit function construction has a uniformly bounded first
derivative, so the mean-value theorem gives the above inequality.

Define the reduced energy
\[
f_{a,U_a}(z)
=
\mathscr E_{a,U_a}
\bigl(z+\Psi_{a,U_a}(z)\bigr).
\]
\medskip
\noindent
\textbf{Step 6. Uniform positivity of the reduced energy.}
\medskip

The goal of this step is to prove that there exist constants
\(c>0\) and \(r_0>0\), independent of \(a\) close to \(a^*\) and
\(U_a\in\mathcal G_a\), such that
\[
f_{a,U_a}(z)
\ge
c\delta_a^2|z|^2
\
\text{for }z\in\mathcal Z_{a,U_a},\ |z|\le r_0.
\]
Recall that \(\mathcal Z_{a,U_a}\) is generated by the radial-translation
and scaling directions.  We first construct an exact deformation whose
derivatives at the origin are precisely these two directions, compute its
energy to leading order, and then compare it with the reduced energy.

\medskip
\noindent
\emph{Step 6.1. Exact radial-translation and scaling deformation.}

Consider
\begin{equation}
\label{eq:sec23-141}
S_{a,U_a}(b,\tau)(x)
=
e^\tau
U_a\bigl(
p_a+e^\tau(x-p_a-b\delta_ae_a)
\bigr).
\end{equation}
This deformation preserves the \(L^2(\mathbb R^2)\)-norm and satisfies
\begin{equation}
\label{eq:sec23-142}
S_{a,U_a}(0,0)=U_a,
\
\partial_bS_{a,U_a}(0,0)=Z_{1,a},
\
\partial_\tau S_{a,U_a}(0,0)=Z_{2,a}.
\end{equation}
Thus the variables \(b\) and \(\tau\) describe, to first order, the
radial-translation and scaling directions identified in Steps 1-2.

Set \(A_a\) as in \eqref{eq:sec23-087}.  With the change of variables
\[
x
=
p_a+b\delta_ae_a+e^{-\tau}(y-p_a),
\]
we obtain the exact identity
\begin{equation}
\label{eq:sec23-143}
E_a^\infty(S_{a,U_a}(b,\tau))
=
e^{2\tau}A_a
+
\int_{\mathbb R^2}
V\left(
p_a+b\delta_ae_a+e^{-\tau}(y-p_a)
\right)
U_a^2(y)\mathrm dy.
\end{equation}

Recall from Lemma~\ref{lem:whole} that
\[
r_a=\frac{|p_a|-1}{\delta_a}\to0.
\]
Rotating the coordinates,
we may assume \(e_a=e_1\).  With
\[
y=p_a+\delta_a\zeta,
\]
we have, uniformly for \((b,\tau)\) in a compact neighborhood of
\((0,0)\) and \(|\zeta|\le R\),
\begin{equation}
\label{eq:sec23-144}
\frac{
\left|
p_a+b\delta_ae_a+e^{-\tau}\delta_a\zeta
\right|-1
}{
\delta_a
}
=
r_a+b+e^{-\tau}\zeta_1
+
O(\delta_a|\zeta|^2).
\end{equation}
Using \(W_a\to q\)  in \(H^1(\mathbb R^2)\),
\eqref{eq:sec23-081} and the exponential estimates in
Lemma~\ref{lem:whole}, we obtain
\[
\begin{aligned}
&\delta_a^{-2}
\int_{\mathbb R^2}
V\left(
p_a+b\delta_ae_a+e^{-\tau}(y-p_a)
\right)
U_a^2(y)\mathrm dy
\\
&\ \to
\int_{\mathbb R^2}
\left(b+e^{-\tau}\zeta_1\right)^2q^2(\zeta)\,d\zeta
\\
&\ =
b^2
+
e^{-2\tau}
\int_{\mathbb R^2}\zeta_1^2q^2(\zeta)\,d\zeta
\\
&\ =
b^2+\gamma e^{-2\tau},
\end{aligned}
\]
where
$
\gamma
=
\int_{\mathbb R^2}\zeta_1^2q^2(\zeta)\,d\zeta
=
\frac{\lambda_0^4}{a^*}.
$
Moreover, by \eqref{eq:sec23-089},
$
A_a=\gamma\delta_a^2+o(\delta_a^2),
\
e_\infty(a)=2\gamma\delta_a^2+o(\delta_a^2).
$
Substituting these relations into \eqref{eq:sec23-143}, we get that
\begin{equation}
\label{eq:sec23-145}
\delta_a^{-2}
\left[
E_a^\infty(S_{a,U_a}(b,\tau))-e_\infty(a)
\right]
\to
\Phi(b,\tau)
\end{equation}
locally uniformly, where
\[
\Phi(b,\tau)
=
b^2
+
\gamma
\left(
e^{-2\tau}+e^{2\tau}-2
\right).
\]

The convergence in \eqref{eq:sec23-145} is in fact local \(C^2\).
Indeed, differentiate \eqref{eq:sec23-143} at most twice with respect to
\(b\) and \(\tau\).  On every bounded \(\zeta\)-region, the required
limits follow from \eqref{eq:sec23-144}.  Near \(x=0\), the possible
\(1/|x|\)-singularity of \(D^2V\) is multiplied by an exponentially small
factor, exactly as in \eqref{eq:sec23-083}.  On the remaining region, the
differentiated integrands are uniformly controlled by the exponential
decay in Lemma~\ref{lem:whole}.  Hence, if
\[
\Phi_{a,U_a}^S(b,\tau)
:=
\delta_a^{-2}
\left[
E_a^\infty(S_{a,U_a}(b,\tau))-e_\infty(a)
\right],
\]
then, for every compact \(K\subset\mathbb R^2\),
\[
\sup_{(b,\tau)\in K}
\left|
\partial^\alpha\Phi_{a,U_a}^S(b,\tau)
-
\partial^\alpha\Phi(b,\tau)
\right|
\to0,
\
|\alpha|\le2.
\]
In particular,
\[
D^2\Phi(0,0)
=
\begin{pmatrix}
2&0\\
0&8\gamma
\end{pmatrix}
>0.
\]
Thus the energy $E_a^\infty(S_{a,U_a}(b,\tau))-e_\infty(a)$ is
strictly positive to second order at the scale \(\delta_a^2\).

\medskip
\noindent
\emph{Step 6.2. Relation with the reduced coordinates.}

The reduction in Step 5 is carried out in
\[
\mathcal X_{a,U_a}
=
\{h\in T_{U_a}:\langle h,Z_{0,a}\rangle=0\},
\]
whereas \(S_{a,U_a}(b,\tau)\) does not necessarily satisfy this
orthogonality condition. We therefore first remove its rotational component
by the modulation constructed in Step 4.

Let $s=s(S_{a,U_a}(b,\tau))$
be the scaled rotation parameter given by Step 4, and define
\[
\widetilde S_{a,U_a}(b,\tau)(x)
:=
S_{a,U_a}(b,\tau)\bigl(R_{\delta_as}x\bigr).
\]
By rotation invariance,
\[
E_a^\infty(\widetilde S_{a,U_a}(b,\tau))
=
E_a^\infty(S_{a,U_a}(b,\tau)).
\]
Moreover,
$\widetilde S_{a,U_a}(0,0)=U_a$
and the mass-chart
coordinate
\[
\xi_{a,U_a}^S(b,\tau)
:=
\mathcal C_{U_a}^{-1}
\bigl(\widetilde S_{a,U_a}(b,\tau)\bigr)
\]
satisfies
$
\xi_{a,U_a}^S(b,\tau)\in\mathcal X_{a,U_a}.
$
Indeed, let \(v=S_{a,U_a}(b,\tau)\), \(s=s(v)\) and \(\theta=\delta_a s\), then
$$
\xi_{a,U_a}^S(b,\tau)(x)
=\left(\frac{v}{\langle v,U_{a,s}\rangle}-U_{a,s}\right)(R_\theta x)
=\Xi(s,v)(R_\theta x).
$$
Hence \(\langle \xi_{a,U_a}^S,U_a\rangle=0\), which, together with \(Z_{0,a,s}(R_\theta x)=Z_{0,a}(x)\) and \eqref{eq:rotation-slice}, implies that

$$
\langle \xi_{a,U_a}^S,Z_{0,a}\rangle
=\langle \Xi(s,v),Z_{0,a,s}\rangle=0.
$$
Therefore \(\xi_{a,U_a}^S(b,\tau)\in\mathcal X_{a,U_a}\).
Decompose
\[
\xi_{a,U_a}^S(b,\tau)
=
z_{a,U_a}^S(b,\tau)
+
h_{a,U_a}^S(b,\tau),
\]
where
\[
z_{a,U_a}^S(b,\tau)\in\mathcal Z_{a,U_a},
\
h_{a,U_a}^S(b,\tau)\in\mathcal H_{a,U_a}.
\]

 Now, we prove that
\[
(b,\tau)\longmapsto z_{a,U_a}^S(b,\tau)
\]
is a uniform local \(C^2\)-diffeomorphism. First, we establish the uniform
\(C^2\)-regularity needed for the inverse function theorem (see, e.g., \cite{AmbrosettiProdi1993,Deimling1985}) and then compute
the differential at the origin.

Specifically, we first claim that the deformation
\[
(b,\tau)\longmapsto S_{a,U_a}(b,\tau)
\]
is uniformly \(C^2\) in the \(\#\)-norm. Indeed,
differentiating at most twice with respect to \(b,\tau\) and then
taking one spatial derivative produces only finitely many polynomially
weighted derivatives of \(U_a\) of order at most three. By the change of variables $x=p_a+\delta_a\zeta,$
these become polynomially weighted derivatives of \(W_a\), which are
uniformly bounded by Lemma~\ref{lem:whole}. The potential part is controlled by
$
V\left(
p_a+\delta_a(be_a+e^{-\tau}\zeta)
\right)
\le
C\delta_a^2(1+|\zeta|^2)
$
for \((b,\tau)\) in a fixed neighborhood of \((0,0)\).
Hence the first and second derivatives of $S_{a,U_a}$ are uniformly bounded in
the \(\#\)-norm and the claim holds. On the other hand,
Step 4 gives the corresponding uniform \(C^2\)-bounds for the
modulation parameter \(s\). Therefore \(\widetilde S_{a,U_a}(b,\tau)\) is also uniformly \(C^2\) in
\((b,\tau)\).
Moreover,
recall that
\[
\xi_{a,U_a}^S(b,\tau)
=
\mathcal C_{U_a}^{-1}\bigl(\widetilde S_{a,U_a}(b,\tau)\bigr)
=
\frac{\widetilde S_{a,U_a}(b,\tau)}
{\langle \widetilde S_{a,U_a}(b,\tau),U_a\rangle}
-
U_a
\]
and
$
\widetilde S_{a,U_a}(0,0)=U_a.
$
Taking $|(b,\tau)|$  sufficiently small if necessary,
\(\langle \widetilde S_{a,U_a}(b,\tau),U_a\rangle\) stays uniformly
away from zero. Hence
\[
(b,\tau)\longmapsto\xi_{a,U_a}^S(b,\tau)
\]
is uniformly \(C^2\). Since the projection
\(P_{\mathcal Z_{a,U_a}}\) is uniformly bounded,
\[
z_{a,U_a}^S
=
P_{\mathcal Z_{a,U_a}}\xi_{a,U_a}^S
\]
is uniformly \(C^2\) as well.

We next compute the differential of $z_{a,U_a}^S(b,\tau)$ at the origin. Set
\[
s_b
:=
\partial_b
\bigl[
s(S_{a,U_a}(b,\tau))
\bigr]\Big|_{(0,0)},
\
s_\tau
:=
\partial_\tau
\bigl[
s(S_{a,U_a}(b,\tau))
\bigr]\Big|_{(0,0)}.
\]
Since
\[
\frac{d}{d\vartheta}R_\vartheta x
\Big|_{\vartheta=0}
=
Jx
\]
and
\[
Z_{0,a}(x)
=
-\delta_a(Jx)\cdot\nabla U_a(x),
\]
we obtain
\[
\begin{aligned}
\partial_b\widetilde S_{a,U_a}(0,0)
&=
\partial_bS_{a,U_a}(0,0)
+
\delta_as_b(Jx)\cdot\nabla U_a(x)=
Z_{1,a}-s_bZ_{0,a}.
\end{aligned}
\]
Similarly,
$
\partial_\tau\widetilde S_{a,U_a}(0,0)
=
Z_{2,a}-s_\tau Z_{0,a}.$
Now, we pass from \(\widetilde S_{a,U_a}\) to its mass-chart coordinate.
For
\[
\mathcal C_{U_a}^{-1}(v)
=
\frac{v}{\langle v,U_a\rangle}-U_a,
\]
a direct differentiation at \(v=U_a\) implies that
\[
D\mathcal C_{U_a}^{-1}(U_a)[w]
=
w-\langle w,U_a\rangle U_a.
\]
Since the deformation preserves the \(L^2(\mathbb R^2)\)-mass,
\[
\partial_b\widetilde S_{a,U_a}(0,0),
\
\partial_\tau\widetilde S_{a,U_a}(0,0)
\in T_{U_a}.
\]
Therefore
$
D\mathcal C_{U_a}^{-1}(U_a)[w]=w
$
for these two vectors, and then
\begin{equation}\label{eqn:partial-Xi}
\partial_b\xi_{a,U_a}^S(0,0)
=
Z_{1,a}-s_bZ_{0,a},\
\partial_\tau\xi_{a,U_a}^S(0,0)
=
Z_{2,a}-s_\tau Z_{0,a}.
\end{equation}
On the other hand,
by the orthogonality condition,
\[
\left\langle
\xi_{a,U_a}^S(b,\tau),Z_{0,a}
\right\rangle
=
0
\]
for all \((b,\tau)\) sufficiently close to \((0,0)\).
Differentiating this identity with respect to \(b\) at \((0,0)\),  we can get that
\[
0
=
\left\langle
\partial_b\xi_{a,U_a}^S(0,0),Z_{0,a}
\right\rangle,
\]
which, together with \eqref{eqn:partial-Xi}, implies that
\[
0
=
\left\langle
Z_{1,a}-s_bZ_{0,a},Z_{0,a}
\right\rangle
=
\langle Z_{1,a},Z_{0,a}\rangle
-
s_b\|Z_{0,a}\|_{L^2(\mathbb R^2)}^2.
\]
Recalling that $\|Z_{0,a}\|_{L^2(\mathbb R^2)}
\to
\|\partial_2q\|_{L^2(\mathbb R^2)}>0,$
we obtain
\[
s_b
=
\frac{\langle Z_{1,a},Z_{0,a}\rangle}
{\|Z_{0,a}\|_{L^2(\mathbb R^2)}^2}.
\]
The same argument with respect to \(\tau\) gives
$
0
=
\left\langle
Z_{2,a}-s_\tau Z_{0,a},Z_{0,a}
\right\rangle,
$
and hence
\[
s_\tau
=
\frac{\langle Z_{2,a},Z_{0,a}\rangle}
{\|Z_{0,a}\|_{L^2(\mathbb R^2)}^2}.
\]
Thus, by the definition of \(\widehat Z_{j,a}\),
\[
\partial_b\xi_{a,U_a}^S(0,0)
=
\widehat Z_{1,a},
\
\partial_\tau\xi_{a,U_a}^S(0,0)
=
\widehat Z_{2,a}.
\]
Since both vectors already belong to \(\mathcal Z_{a,U_a}\),
applying \(P_{\mathcal Z_{a,U_a}}\) yields
\begin{equation}
\label{eq:reduced-chart-derivative}
Dz_{a,U_a}^S(0,0)[\beta_1,\beta_2]
=
\beta_1\widehat Z_{1,a}
+
\beta_2\widehat Z_{2,a}.
\end{equation}
Under the blow-up scaling,
\[
\widehat Z_{1,a}\to-\partial_1q,
\
\widehat Z_{2,a}\to\Lambda q
\]
 in \(H^1(\mathbb R^2)\).
Since the two limiting functions are linearly independent,
their Gram matrix is positive definite. Hence there exist constants
\(c,C>0\), independent of \(a\) and \(U_a\), such that
\[
c(\beta_1^2+\beta_2^2)
\le
\|\beta_1\widehat Z_{1,a}
+\beta_2\widehat Z_{2,a}\|_{*,a}^2
\le
C(\beta_1^2+\beta_2^2).
\]
Therefore
\[
Dz_{a,U_a}^S(0,0):
\mathbb R^2\to\mathcal Z_{a,U_a}
\]
is an isomorphism whose inverse norm is uniformly bounded.

Finally, the inverse function theorem yields there are
\[
r_{\rm par},r_{\rm reduced}>0,
\]
independent of \(a\) and \(U_a\), such that
\[
(b,\tau)\longmapsto z_{a,U_a}^S(b,\tau)
\]
is a \(C^2\)-diffeomorphism from \(B_{r_{\rm par}}(0)\) onto its image.
Its local inverse has uniformly bounded \(C^1\)-norm, and
\[
\mathbb B_{r_{\rm reduced}}^{\mathcal Z}(0)
\subset
z_{a,U_a}^S\bigl(B_{r_{\rm par}}(0)\bigr).
\]

\medskip
\noindent
\emph{Step 6.3. Comparison with the Lyapunov-Schmidt correction.}

For brevity, let
\[
p=(b,\tau),
\
z^S(p)=z_{a,U_a}^S(b,\tau),
\
h^S(p)=h_{a,U_a}^S(b,\tau),
\]
and set
$
r:=|b|+|\tau|.
$
The mass-chart coordinate of the exact deformation is
$\xi_{a,U_a}^S(p)=z^S(p)+h^S(p),$
whereas the Lyapunov-Schmidt graph over the same reduced coordinate is
\[
z^S(p)+\Psi_{a,U_a}(z^S(p)).
\]
Accordingly, define the difference between the two
\(\mathcal H_{a,U_a}\)-components by
\[
d(p)
:=
\Psi_{a,U_a}(z^S(p))-h^S(p)
\in\mathcal H_{a,U_a},
\]
and define the residual of the exact deformation by
\[
g(p)
:=
\mathcal N_{a,U_a}(z^S(p),h^S(p))
\in\mathcal H_{a,U_a}^*.
\]
The argument proceeds in four stages: we first estimate the residual
\(g(p)\), then solve for \(d(p)\), next differentiate the equation for
\(d(p)\), and finally compare the two energies.

\smallskip
\noindent
\emph{(a) The residual in the complementary equation.}

If
$
y
=
p_a+e^\tau(x-p_a-b\delta_ae_a),
$
then the equation satisfied by \(U_a\) gives
\begin{equation}
\label{eq:sec23-147}
\begin{aligned}
&-\Delta S_{a,U_a}
+
VS_{a,U_a}
-
aS_{a,U_a}^3
-
\mu_a^\infty e^{2\tau}S_{a,U_a}
=
\left[
V(x)-e^{2\tau}V(y)
\right]S_{a,U_a}.
\end{aligned}
\end{equation}
If
$
\eta\in H_V^1(\mathbb R^2),
\
\langle\eta,S_{a,U_a}\rangle=0,
$
then the Lagrange-multiplier term vanishes after pairing
\eqref{eq:sec23-147} with \(\eta\). Hence
\[
DE_a^\infty(S_{a,U_a})[\eta]
=
2\int_{\mathbb R^2}
\left[
V(x)-e^{2\tau}V(y)
\right]
S_{a,U_a}(x)\eta(x)\,\mathrm dx.
\]
Introduce the blow-up variable \(\zeta\) by
\[
y=p_a+\delta_a\zeta,
\
x=p_a+\delta_a(be_a+e^{-\tau}\zeta).
\]
For \(r\) sufficiently small,
\[
V(x)-e^{2\tau}V(y)
=
O\left(
\delta_a^2r(1+|\zeta|^2)
\right)
\]
on bounded \(\zeta\)-sets.
Define
$
\widetilde\eta(\zeta)
:=
\delta_a
\eta\left(
p_a+\delta_a(be_a+e^{-\tau}\zeta)
\right).
$
If \(\|\eta\|_{\#,a}=1\), then
$
\|\widetilde\eta\|_{H^1(\mathbb R^2)}
\le C.
$
Using the exponential decay in Lemma~\ref{lem:whole}, we obtain
\[
\begin{aligned}
|DE_a^\infty(S_{a,U_a})[\eta]|
&\le
C\delta_a^2r
\int_{\mathbb R^2}
(1+|\zeta|^2)e^{-\sigma|\zeta|}
|\widetilde\eta(\zeta)|\,\mathrm d\zeta
\\
&\le
C\delta_a^2r.
\end{aligned}
\]
Consequently,
\begin{equation}
\label{eq:sec23-148}
\sup_{\{
\eta\in H_V^1(\mathbb R^2):
\|\eta\|_{\#,a}=1,
\langle\eta,S_{a,U_a}\rangle=0
\}}
\left|
DE_a^\infty(S_{a,U_a}(b,\tau))[\eta]
\right|
\le
C\delta_a^2r.
\end{equation}
By \eqref{eq:sec23-148}, together with rotation invariance and the
uniform \(C^2\)-regularity of the mass-chart established in Step~6.2, we have
\begin{equation}
\label{eqn:g-P}
\|g(p)\|
\le
C\delta_a^2r.
\end{equation}

We also need to estimate the derivatives of the residual with respect to the parameter \(p=(b,\tau)\) up to second order. In the
rescaled variables, set
\[
B_a(p,\zeta)
:=
V\left(
p_a+\delta_a(be_a+e^{-\tau}\zeta)
\right)
-
e^{2\tau}V(p_a+\delta_a\zeta).
\]
For every multi-index \(\alpha\) with \(1\le|\alpha|\le2\), differentiation
with respect to \(p=(b,\tau)\) gives, on bounded \(\zeta\)-sets,
\[
|\partial_p^\alpha B_a(p,\zeta)|
\le
C\delta_a^2(1+|\zeta|^2).
\]
The remaining region, including the possible singularity of \(D^2V\) at
the origin, is controlled by the exponential estimates in
Lemma~\ref{lem:whole}.
Together with the uniform \(C^2\)-bounds from
Step~6.2, this yields
\begin{equation}
\label{eqn:g-P-derivatives}
\|\partial_p^\alpha g(p)\|
\le
C\delta_a^2,
\
1\le|\alpha|\le2.
\end{equation}
Moreover, Step 6.2 gives
\begin{equation}
\label{eqn:exact-coordinate-derivatives}
\|\partial_p^\alpha z^S(p)\|_{\#,a}
+
\|\partial_p^\alpha h^S(p)\|_{\#,a}
\le
C,
\
1\le|\alpha|\le2.
\end{equation}
Since \(z^S(0)=h^S(0)=0\), the mean-value theorem yields
\begin{equation}
\label{eqn:exact-coordinate-size}
|z^S(p)|
+
\|h^S(p)\|_{\#,a}
\le
Cr.
\end{equation}
By \eqref{eqn:Psi-bound} and \eqref{eqn:exact-coordinate-size},
\begin{equation}
\label{eqn:Psi-size-along-exact}
\|\Psi_{a,U_a}(z^S(p))\|_{\#,a}
\le
C|z^S(p)|
\le
Cr.
\end{equation}

\smallskip
\noindent
\emph{(b) Solving for the difference \(d(p)\).}

For \(0\le t\le1\), set
$
h_t(p)
:=
h^S(p)+t\,d(p).
$
Thus
\[
h_0(p)=h^S(p),
\
h_1(p)=\Psi_{a,U_a}(z^S(p)).
\]
By \eqref{eqn:exact-coordinate-size} and
\eqref{eqn:Psi-size-along-exact},
\[
|z^S(p)|+\|h_t(p)\|_{\#,a}
\le
Cr
\]
uniformly for \(t\in[0,1]\).
Recall that
$
A_{a,U_a}
=
D_h\mathcal N_{a,U_a}(0,0).
$
By \eqref{eq:sec23-136},
$\|A_{a,U_a}^{-1}\|
\le
C\delta_a^2.$ Then
\eqref{eq:sec23-138} gives
\begin{equation}
\label{eqn:D-N-A}
\left\|
D_h\mathcal N_{a,U_a}(z^S,h_t)
-
A_{a,U_a}
\right\|
\le
C\delta_a^{-2}
\left(
|z^S|+\|h_t\|_{\#,a}
\right)
\le
C\delta_a^{-2}r.
\end{equation}
Consequently,
\[
\left\|
A_{a,U_a}^{-1}
\left[
D_h\mathcal N_{a,U_a}(z^S,h_t)-A_{a,U_a}
\right]
\right\|
\le
Cr.
\]
Taking $r>0$ sufficiently small so that
 \(Cr\le1/2\),
the standard Neumann-series criterion shows that
\(D_h\mathcal N_{a,U_a}(z^S,h_t)\) is invertible for every
\(t\in[0,1]\), with
\begin{equation}
\label{eqn:segment-inverse}
\left\|
D_h\mathcal N_{a,U_a}(z^S,h_t)^{-1}
\right\|
\le
C\delta_a^2.
\end{equation}

To compare the two endpoints of $h_t(p)$, define the averaged operator
\[
\overline L(p)
:=
\int_0^1
D_h\mathcal N_{a,U_a}
\left(
z^S(p),h^S(p)+td(p)
\right)
\,\mathrm dt.
\]
By \eqref{eqn:D-N-A},
$
\|\overline L(p)-A_{a,U_a}\|
\le
C\delta_a^{-2}r,
$
and hence
\[
\left\|
A_{a,U_a}^{-1}
\left(
\overline L(p)-A_{a,U_a}
\right)
\right\|
\le
Cr
\le
\frac12.
\]
Thus \(\overline L(p)\) is invertible and
\begin{equation}
\label{eqn:averaged-inverse}
\|\overline L(p)^{-1}\|
\le
C\delta_a^2.
\end{equation}

Note that
\[
\mathcal N_{a,U_a}
\left(
z^S(p),\Psi_{a,U_a}(z^S(p))
\right)
=0,
\
\mathcal N_{a,U_a}(z^S(p),h^S(p))
=g(p).
\]
By the Banach-space fundamental theorem of calculus, we have
\[
\begin{aligned}
-g(p)
&=
\mathcal N_{a,U_a}(z^S,h^S+d)
-
\mathcal N_{a,U_a}(z^S,h^S)
\\
&=
\int_0^1
D_h\mathcal N_{a,U_a}
\left(
z^S,h^S+td
\right)[d]
\,\mathrm dt
\\
&=
\overline L(p)d(p).
\end{aligned}
\]
Combining this identity with \eqref{eqn:g-P} and
\eqref{eqn:averaged-inverse}, we obtain
\begin{equation}
\label{eqn:d-P}
\|d(p)\|_{\#,a}
\le
C\delta_a^4r.
\end{equation}
Equivalently,
\[
\left\|
\Psi_{a,U_a}(z_{a,U_a}^S(b,\tau))
-
h_{a,U_a}^S(b,\tau)
\right\|_{\#,a}
\le
C\delta_a^4r.
\]

\smallskip
\noindent
\emph{(c) Parameter derivatives of \(d(p)\).}

Before differentiating the equation for \(d(p)\), we record the uniform
bounds on the derivatives of $\mathscr E_{a,U_a}$ that will be used below.
For \(k\ge2\), the norm \(\|D^k\mathscr E_{a,U_a}(\xi)\|\) denotes the
operator norm of the \(k\)-linear form with respect to the \(\#\)-norm in
each argument.

Recall from Step 5 that, for \(\xi\in T_{U_a}\),
\[
\mathscr E_{a,U_a}(\xi)
=
\frac{\mathcal A(U_a+\xi)}{1+\|\xi\|_{L^2(\mathbb R^2)}^2}
-
\frac a2
\frac{\mathcal Q(U_a+\xi)}{
\bigl(1+\|\xi\|_{L^2(\mathbb R^2)}^2\bigr)^2}
-
e_\infty(a),
\]
where
\[
\mathcal A(w)
=
\int_{\mathbb R^2}
\bigl(|\nabla w|^2+Vw^2\bigr),
\
\mathcal Q(w)
=
\int_{\mathbb R^2}|w|^4.
\]
Set
$R(\xi):=1+\|\xi\|_{L^2(\mathbb R^2)}^2.$
On a fixed small \(\#\)-ball in the sense that $\|\xi\|_{\#,a}$ is small, \(R(\xi)\) is bounded above and below by
positive constants, and the derivatives of \(R^{-1}\) and \(R^{-2}\) up to
order four are uniformly bounded. This follows from
\[
DR(\xi)[\eta]=2\langle\xi,\eta\rangle,
\
D^2R(\xi)[\eta_1,\eta_2]
=2\langle\eta_1,\eta_2\rangle,
\]
with all higher derivatives of \(R\) equal to zero.

For the quadratic part, if
$
\mathcal A(\eta_1,\eta_2)
=
\int_{\mathbb R^2}
\bigl(
\nabla\eta_1\cdot\nabla\eta_2
+V\eta_1\eta_2
\bigr),
$
then by the definition of the \(\#\)-norm, we infer that
\[
|\mathcal A(\eta_1,\eta_2)|
\le
\delta_a^{-2}
\|\eta_1\|_{\#,a}\|\eta_2\|_{\#,a}.
\]
Moreover, Lemma~\ref{lem:whole} gives
$
\mathcal A(U_a)\le C\delta_a^{-2}.
$
Since \(\mathcal A(U_a+\xi)\) is quadratic in \(\xi\), these two bounds
control the function itself and all of its nonzero derivatives by
\(C\delta_a^{-2}\) on a fixed small \(\#\)-ball.
For the quartic part, the derivatives have the explicit form
\[
D^j\mathcal Q(U_a+\xi)
[\eta_1,\ldots,\eta_j]
=
\frac{4!}{(4-j)!}
\int_{\mathbb R^2}
(U_a+\xi)^{4-j}
\eta_1\cdots\eta_j,
\
0\le j\le4.
\]
By \eqref{eq:sec23-103} and Lemma~\ref{lem:whole},
\[
\|\eta\|_{L^4(\mathbb R^2)}
\le
C\delta_a^{-1/2}\|\eta\|_{\#,a},
\
\|U_a\|_{L^4(\mathbb R^2)}
\le
C\delta_a^{-1/2}.
\]
Hence, on a fixed small \(\#\)-ball,
\[
\|U_a+\xi\|_{L^4(\mathbb R^2)}
\le
C\delta_a^{-1/2},
\]
and H\"older's inequality yields
\[
\left|
D^j\mathcal Q(U_a+\xi)
[\eta_1,\ldots,\eta_j]
\right|
\le
C\delta_a^{-2}
\prod_{\ell=1}^j
\|\eta_\ell\|_{\#,a},
\
0\le j\le4.
\]
Therefore
\begin{equation}
\label{eqn:energy-C4-bound}
\|D^k\mathscr E_{a,U_a}(\xi)\|
\le
C\delta_a^{-2},
\
2\le k\le4,
\end{equation}
uniformly on a fixed small \(\#\)-ball. Since
\[
\mathcal N_{a,U_a}(z,h)[\eta]
=
D\mathscr E_{a,U_a}(z+h)[\eta],
\]
\eqref{eqn:energy-C4-bound} also implies
\begin{equation}
\label{eqn:N-C3-bound}
\|D_{(z,h)}^j\mathcal N_{a,U_a}(z,h)\|
\le
C\delta_a^{-2},
\
1\le j\le3,
\end{equation}
when \(|z|+\|h\|_{\#,a}\) is sufficiently small.

Now, we define
\[
\mathcal F(p,d)
:=
\mathcal N_{a,U_a}
\left(
z^S(p),h^S(p)+d
\right),
\]
then
\[
\mathcal F(p,d(p))=0,
\
\mathcal F(p,0)=g(p).
\]
Set
\[
L(p)
:=
D_d\mathcal F(p,d(p))
=
D_h\mathcal N_{a,U_a}
\left(
z^S(p),\Psi_{a,U_a}(z^S(p))
\right).
\]
By \eqref{eqn:segment-inverse},
\begin{equation}
\label{eqn:L-inverse}
\|L(p)^{-1}\|
\le
C\delta_a^2.
\end{equation}
Let \(i\in\{b,\tau\}\). Differentiating
\(\mathcal F(p,d(p))=0\) with respect to \(p_i\), we have
\[
L(p)\partial_i d(p)
=
-\partial_i\mathcal F(p,d(p)).
\]
Since \(\partial_i\mathcal F(p,0)=\partial_i g(p)\), the mean-value theorem
in the \(d\)-variable, together with \eqref{eqn:N-C3-bound} and
\eqref{eqn:exact-coordinate-derivatives}, states that
\[
\begin{aligned}
\|\partial_i\mathcal F(p,d(p))-\partial_i g(p)\|
\le
C\delta_a^{-2}\|d(p)\|_{\#,a}\le
C\delta_a^2r.
\end{aligned}
\]
Using \eqref{eqn:g-P-derivatives}, we get that
$
\|\partial_i\mathcal F(p,d(p))\|
\le
C\delta_a^2.
$
Together with \eqref{eqn:L-inverse}, this yields
\begin{equation}
\label{eqn:d-first-derivatives}
\|\partial_i d(p)\|_{\#,a}
\le
C\delta_a^4,
\
i\in\{b,\tau\}.
\end{equation}
For the second derivatives, let \(i,j\in\{b,\tau\}\). Differentiating once
more, we deduce that
\[
\begin{aligned}
L(p)\partial_{ij}d(p)
&=
-\partial_{ij}\mathcal F(p,d(p))
-
D_d\partial_i\mathcal F(p,d(p))[\partial_jd(p)]
 -
D_d\partial_j\mathcal F(p,d(p))[\partial_id(p)]
\\
&\ \ -
D_d^2\mathcal F(p,d(p))
[\partial_id(p),\partial_jd(p)].
\end{aligned}
\]
Since \(\partial_{ij}\mathcal F(p,0)=\partial_{ij}g(p)\),
\eqref{eqn:exact-coordinate-derivatives}, \eqref{eqn:d-P} and \eqref{eqn:N-C3-bound}
imply that
\[
\|\partial_{ij}\mathcal F(p,d(p))-\partial_{ij}g(p)\|
\le
C\delta_a^{-2}\|d(p)\|_{\#,a}
\le
C\delta_a^2r.
\]
Thus, by \eqref{eqn:g-P-derivatives}, $\|\partial_{ij}\mathcal F(p,d(p))\|
\le
C\delta_a^2.$
Furthermore, \eqref{eqn:N-C3-bound} gives
\[
\|D_d\partial_i\mathcal F(p,d(p))\|
+
\|D_d^2\mathcal F(p,d(p))\|
\le
C\delta_a^{-2}.
\]
Using \eqref{eqn:d-first-derivatives}, we conclude that
$\|L(p)\partial_{ij}d(p)\|
\le
C\delta_a^2.$
Finally, \eqref{eqn:L-inverse} implies
\[
\|\partial_{ij}d(p)\|_{\#,a}
\le
C\delta_a^4.
\]
Combining the first- and second-derivative estimates, we obtain
\begin{equation}
\label{eqn:d-parameter-derivatives}
\left\|
\partial_{(b,\tau)}^\alpha
\left[
\Psi_{a,U_a}(z_{a,U_a}^S(b,\tau))
-
h_{a,U_a}^S(b,\tau)
\right]
\right\|_{\#,a}
\le
C\delta_a^4,
\
1\le|\alpha|\le2.
\end{equation}

\smallskip
\noindent
\emph{(d) Comparison of the two energies.}

Rotation invariance and the definition of the mass-chart give
\[
E_a^\infty(S_{a,U_a}(p))-e_\infty(a)
=
\mathscr E_{a,U_a}(z^S(p)+h^S(p)).
\]
On the other hand,
\[
f_{a,U_a}(z^S(p))
=
\mathscr E_{a,U_a}
\left(
z^S(p)+h^S(p)+d(p)
\right).
\]
Define
\[
\mathcal R_{a,U_a}(p)
:=
f_{a,U_a}(z^S(p))
-
\left[
E_a^\infty(S_{a,U_a}(p))-e_\infty(a)
\right].
\]
Since \(d(p)\in\mathcal H_{a,U_a}\), by Taylor's formula in the
\(\mathcal H_{a,U_a}\)-variable, we obtain
\[
\begin{aligned}
\mathcal R_{a,U_a}(p)
&=
D\mathscr E_{a,U_a}(z^S+h^S)[d]+
\int_0^1(1-t)
D^2\mathscr E_{a,U_a}
(z^S+h^S+td)[d,d]
\,\mathrm dt.
\end{aligned}
\]
Because \(d(p)\in\mathcal H_{a,U_a}\), the first term is
$
D\mathscr E_{a,U_a}(z^S+h^S)[d]
=
g(p)[d(p)].
$
Thus, by \eqref{eqn:g-P}, \eqref{eqn:d-P}, and
\eqref{eqn:energy-C4-bound},
\[
\begin{aligned}
|\mathcal R_{a,U_a}(p)|
&\le
\|g(p)\|_{\mathcal H_{a,U_a}^*}
\|d(p)\|_{\#,a}
+
C\delta_a^{-2}\|d(p)\|_{\#,a}^2
\\
&\le
C\delta_a^2r\cdot C\delta_a^4r
+
C\delta_a^{-2}(C\delta_a^4r)^2
\\
&\le
C\delta_a^6r^2.
\end{aligned}
\]
Therefore
\begin{equation}
\label{eq:sec23-150}
\left|
f_{a,U_a}(z_{a,U_a}^S(b,\tau))
-
\left[
E_a^\infty(S_{a,U_a}(b,\tau))-e_\infty(a)
\right]
\right|
\le
C\delta_a^6r^2.
\end{equation}

It remains to control the derivatives of the remainder with respect to \(p=(b,\tau)\) up to second order.
Write
\[
\mathcal R_{a,U_a}(p)=g(p)[d(p)]+Q(p),
\]
where
\[
Q(p)
:=
\int_0^1(1-t)
D^2\mathscr E_{a,U_a}
(z^S+h^S+td)[d,d]
\,\mathrm dt.
\]
For the linear term,
$
\partial_i(g[d])
=(\partial_i g)[d]+g[\partial_i d].
$
Hence, by \eqref{eqn:g-P}, \eqref{eqn:g-P-derivatives},
\eqref{eqn:d-P} and \eqref{eqn:d-first-derivatives},
\[
|\partial_i(g[d])|
\le
C\delta_a^6r.
\]
Similarly,
\[
\begin{aligned}
\partial_{ij}(g[d])
&=
(\partial_{ij}g)[d]
+(\partial_i g)[\partial_jd]
+(\partial_j g)[\partial_id]
+g[\partial_{ij}d],
\end{aligned}
\]
so \eqref{eqn:d-parameter-derivatives} gives
$
|\partial_{ij}(g[d])|
\le
C\delta_a^6.
$
For the quadratic term, set
\[
X_t(p):=z^S(p)+h^S(p)+td(p).
\]
Note that
\[
\begin{aligned}
\partial_iQ(p)
&=
\int_0^1(1-t)
\Bigl(
D^3\mathscr E_{a,U_a}(X_t)
[\partial_iX_t,d,d]
+
2D^2\mathscr E_{a,U_a}(X_t)
[\partial_id,d]
\Bigr)
\,\mathrm dt.
\end{aligned}
\]
Here
$\|\partial_iX_t\|_{\#,a}\le C$
by \eqref{eqn:exact-coordinate-derivatives} and
\eqref{eqn:d-first-derivatives}. Therefore, \eqref{eqn:d-P},
\eqref{eqn:energy-C4-bound} and
\eqref{eqn:d-first-derivatives} imply
\[
|\partial_iQ(p)|
\le
C\delta_a^6r.
\]
Differentiating once more and using the same bounds, we deduce that
\[
\begin{aligned}
|\partial_{ij}Q(p)|
&\le C\delta_a^{-2}\Bigl[
\bigl(
\|\partial_{ij}X_t\|_{\#,a}
+
\|\partial_iX_t\|_{\#,a}
\|\partial_jX_t\|_{\#,a}
\bigr)
\|d\|_{\#,a}^2
\\
&\ \ +
\bigl(
\|\partial_iX_t\|_{\#,a}\|\partial_jd\|_{\#,a}
+
\|\partial_jX_t\|_{\#,a}\|\partial_id\|_{\#,a}
\bigr)
\|d\|_{\#,a}
\\
&\ \ +
\|\partial_id\|_{\#,a}
\|\partial_jd\|_{\#,a}
+
\|\partial_{ij}d\|_{\#,a}
\|d\|_{\#,a}
\Bigr].
\end{aligned}
\]
The right-hand side is bounded by \(C\delta_a^6\), because
\eqref{eqn:exact-coordinate-derivatives} and
\eqref{eqn:d-parameter-derivatives} imply
\[
\|\partial_iX_t\|_{\#,a}
+
\|\partial_{ij}X_t\|_{\#,a}
\le C.
\]
Consequently,
\[
|\partial_{ij}Q(p)|
\le
C\delta_a^6.
\]
Combining the above estimates for \(g[d]\) and \(Q\), we obtain
\begin{equation}
\label{eqn:remainder-C2-bounds}
|\nabla_{(b,\tau)}\mathcal R_{a,U_a}(b,\tau)|
\le
C\delta_a^6r,
\
\|D^2_{(b,\tau)}\mathcal R_{a,U_a}(b,\tau)\|
\le
C\delta_a^6.
\end{equation}
Together with \eqref{eq:sec23-150}, this shows that, for every compact
set \(K\) contained in the fixed small neighborhood of $(0,0)$, there is $C_K>0$,
\[
\|\delta_a^{-2}\mathcal R_{a,U_a}\|_{C^2(K)}
\le
C_K\delta_a^4
\to0.
\]
Thus
\begin{equation}
\label{eqn:remainder-C2-convergence}
\delta_a^{-2}\mathcal R_{a,U_a}
\to0
\
\text{locally in }C^2.
\end{equation}
Combining \eqref{eqn:remainder-C2-convergence} with the local \(C^2\)
convergence in \eqref{eq:sec23-145}, we conclude that
\begin{equation}
\label{eqn:reduced-energy-C2}
\delta_a^{-2}
f_{a,U_a}
\left(
z_{a,U_a}^S(b,\tau)
\right)
\to
\Phi(b,\tau)
\
\text{locally in }C^2.
\end{equation}

\medskip
\noindent
\emph{Step 6.4. Positivity of the reduced energy.}

Define
\[
F_{a,U_a}(b,\tau)
:=
\delta_a^{-2}
f_{a,U_a}
\left(
z_{a,U_a}^S(b,\tau)
\right).
\]
Then $F_{a,U_a}(0,0)=0$ and
$DF_{a,U_a}(0,0)=0$
since
\[
z_{a,U_a}^S(0,0)=0,
\
f_{a,U_a}(0)=0,
\
Df_{a,U_a}(0)=0.
\]
Note that
\[
D^2\Phi(0,0)
=
\begin{pmatrix}
2&0\\
0&8\gamma
\end{pmatrix}.
\]
Set
\[
\lambda_*
:=
\min\{2,8\gamma\}>0.
\]
By continuity of \(D^2\Phi\) and \eqref{eqn:reduced-energy-C2}, there exists
$r_{\rm pos}>0,$
independent of \(a\) and \(U_a\), such that
\[
v^TD^2\Phi(b,\tau)v
\ge
\frac{3\lambda_*}{4}
|v|_{\mathbb R^2}^2
\]
when $b^2+\tau^2\le r_{\rm pos}^2.$
By the local \(C^2\)-convergence,
\[
\left\|
D^2F_{a,U_a}(b,\tau)
-
D^2\Phi(b,\tau)
\right\|
\le
\frac{\lambda_*}{4}
\]
when $b^2+\tau^2\le r_{\rm pos}^2,$ uniformly in \(U_a\). Hence
\[
v^TD^2F_{a,U_a}(b,\tau)v
\ge
\frac{\lambda_*}{2}
|v|_{\mathbb R^2}^2.
\]
Set $p=(b,\tau)$, then by
Taylor's formula along the segment \(tp\) with \(0\le t\le1\), we arrive at
\[
F_{a,U_a}(p)
=
\int_0^1
(1-t)
D^2F_{a,U_a}(tp)[p,p]
\,\mathrm dt.
\]
Therefore
\[
\begin{aligned}
F_{a,U_a}(b,\tau)
\ge
\frac{\lambda_*}{2}
(b^2+\tau^2)
\int_0^1(1-t)\,\mathrm dt
=
\frac{\lambda_*}{4}
(b^2+\tau^2).
\end{aligned}
\]
Consequently,
we obtain for $b^2+\tau^2\le r_{\rm pos}^2$,
\begin{equation}
\label{eqn:parameter-reduced-positivity}
f_{a,U_a}
\left(
z_{a,U_a}^S(b,\tau)
\right)
\ge
\frac{\lambda_*}{4}\delta_a^2(b^2+\tau^2).
\end{equation}

Let
$\Theta_{a,U_a}$
denote the local inverse of
\[
(b,\tau)
\longmapsto
z_{a,U_a}^S(b,\tau).
\]
By Step 6.2,
the inverse function theorem gives uniform constants
\(C_z,C_\Theta>0\) such that
\[
\sup_{|p|\le r_{\rm par}}
\|Dz_{a,U_a}^S(p)\|
\le
C_z
\]
and
\[
\sup_{|z|\le r_{\rm reduced}}
\|D\Theta_{a,U_a}(z)\|
\le
C_\Theta.
\]
Since
$z_{a,U_a}^S(0)=0,$ then by
the mean-value theorem,
\[
\begin{aligned}
|z_{a,U_a}^S(p)|
=
|z_{a,U_a}^S(p)-z_{a,U_a}^S(0)|
\le
C_z|p|_{\mathbb R^2}.
\end{aligned}
\]
If
\[
z=z_{a,U_a}^S(p),
\
p=\Theta_{a,U_a}(z),
\]
then
\[
|z|
\le
C_z|\Theta_{a,U_a}(z)|_{\mathbb R^2}.
\]
Thus
\[
|\Theta_{a,U_a}(z)|_{\mathbb R^2}
\ge
C_z^{-1}|z|.
\]
On the other hand, $\Theta_{a,U_a}(0)=0,$
and therefore
\[
\begin{aligned}
|\Theta_{a,U_a}(z)|_{\mathbb R^2}
=
|\Theta_{a,U_a}(z)-\Theta_{a,U_a}(0)|_{\mathbb R^2}
\le
C_\Theta|z|.
\end{aligned}
\]
Hence
\begin{equation}
\label{eqn:theta-bi-Lipschitz}
C_z^{-1}|z|
\le
|\Theta_{a,U_a}(z)|_{\mathbb R^2}
\le
C_\Theta|z|.
\end{equation}

It remains to choose the reduced-coordinate radius so that the inverse map stays inside the parameter region where \eqref{eqn:parameter-reduced-positivity} holds.
Choose $r_{0}>0$
so small that
$r_{0}
\le
r_{\rm reduced}$
and
$
C_\Theta r_{0}
\le
r_{\rm pos}.
$
Then, when $|z|\le r_{0},$
we have
$|\Theta_{a,U_a}(z)|_{\mathbb R^2}
\le
r_{\rm pos},$
so the preceding positivity estimate \eqref{eqn:parameter-reduced-positivity} in the \((b,\tau)\)-variables
is applicable.

Taking $(b,\tau)
=\Theta_{a,U_a}(z),$
and using \eqref{eqn:parameter-reduced-positivity} and
\eqref{eqn:theta-bi-Lipschitz}, we obtain
\[
\begin{aligned}
f_{a,U_a}(z)
&=
f_{a,U_a}
\left(
z_{a,U_a}^S(\Theta_{a,U_a}(z))
\right)
\\
&\ge
\frac{\lambda_*}{4}\delta_a^2
|\Theta_{a,U_a}(z)|_{\mathbb R^2}^2
\\
&\ge
\frac{\lambda_*}{4}C_z^{-2}
\delta_a^2|z|^2.
\end{aligned}
\]
Finally, there is $c>0$ such that
\begin{equation}
\label{eq:sec23-152}
f_{a,U_a}(z)
\ge
c\delta_a^2|z|^2,\
|z|\le r_0.
\end{equation}
This proves the required uniform positivity of the reduced energy.

\medskip
\noindent
\textbf{Step 7. Quadratic growth in the complementary direction.}

 Let \(v\in H_V^1(\mathbb R^2)\) satisfy
$\|v\|_{L^2(\mathbb R^2)}=1,$
and set
\[
d_{\#,a}(v)
:=
\operatorname{dist}_{\#,a}(v,\mathcal G_a).
\]
We may assume that \(d_{\#,a}(v)>0\) is sufficiently small and choose
\(\widehat U_a\in\mathcal G_a\) such that
\begin{equation}
\label{eq:choose-nearby-ground-state}
\|v-\widehat U_a\|_{\#,a}
\le
2d_{\#,a}(v).
\end{equation}
In particular,
$\|v-\widehat U_a\|_{*,a}
\le
2d_{\#,a}(v).$
By
Step~4, there exists a unique small \(s=s(v)\) such that
\[
\widehat U_{a,s}(x)
:=
\widehat U_a(R_{-\delta_as}x)
\ \text{and}\
\xi
:=
\Xi(s(v),v)
\]
satisfy
\[
v
=
\mathcal C_{\widehat U_{a,s(v)}}(\xi)\ \text{and}\
\langle\xi,\widehat U_{a,s(v)}\rangle
=
\langle\xi,Z_{0,a,s(v)}\rangle
=
0.
\]
Moreover, $|s(v)|
\le
C\|v-\widehat U_a\|_{*,a}$ for some constant $C>0$.
Note that
$$
\widehat U_{a,s}-\widehat U_a=\int_{0}^s \partial_t\widehat U_{a,t}\mathrm dt
$$
and $s=s(v)$ is bounded. By Lemma~\ref{lem:whole}, we know that
\begin{equation*}
\|
\widehat U_{a,s(v)}-\widehat U_a
\|_{\#,a}
\le
C|s(v)|
\le
C\|v-\widehat U_a\|_{\#,a}.
\end{equation*}
Consequently,
\begin{equation}
\label{eq:v-modulated-basepoint-control}
\|
v-\widehat U_{a,s(v)}
\|_{\#,a}
\le
C\|v-\widehat U_a\|_{\#,a}.
\end{equation}
Since
\[
1-\langle
v,\widehat U_{a,s(v)}
\rangle
=
\frac12
\|
v-\widehat U_{a,s(v)}
\|_{L^2(\mathbb R^2)}^2,
\]
then \eqref{eq:choose-nearby-ground-state} and  \eqref{eq:v-modulated-basepoint-control} imply that
 $\langle
v,\widehat U_{a,s(v)}
\rangle\geq\frac{1}{2}$ provided that \(d_{\#,a}(v)>0\) is sufficiently small.
Then
\[
\xi
=
\frac{v-\widehat U_{a,s(v)}}{\langle
v,\widehat U_{a,s(v)}
\rangle}
+
\left(\frac{1}{\langle
v,\widehat U_{a,s(v)}
\rangle}-1\right)
\widehat U_{a,s(v)},
\]
which, together with
rotation invariance of the \(\#\)-norm,  the bound
$\sup_{U_a\in\mathcal G_a}\|U_a\|_{\#,a}\le C$
and \eqref{eq:v-modulated-basepoint-control}, implies that
\begin{equation}
\label{eq:modulated-coordinate-control}
\|\xi\|_{\#,a}
\le
C\|v-\widehat U_a\|_{\#,a}
\le
C d_{\#,a}(v).
\end{equation}

Recalling that \(\widehat U_{a,s(v)}\in\mathcal G_a\), we now relabel
\[
U_a
:=
\widehat U_{a,s(v)}\
\text{and}\
Z_{0,a}
:=
Z_{0,a,s(v)}.
\]
Then $v=\mathcal C_{U_a}(\xi)$
with
$\xi\in\mathcal X_{a,U_a}$.
Since all the constructions and estimates in Steps~5-6 are uniform in
\(U_a\in\mathcal G_a\), they apply to this choice of \(U_a\).
Let
$z:=P_{\mathcal Z_{a,U_a}}\xi,$ then
the uniform boundedness of \(P_{\mathcal Z_{a,U_a}}\),
\eqref{eq:sec23-131} and \eqref{eq:modulated-coordinate-control} give
$$
|z|
\le
C\|z\|_{\#,a}
\le
C\|\xi\|_{\#,a}\le
C d_{\#,a}(v).
$$
Using the splitting from Step~5,  $\xi$ can in turn be
written uniquely as
\[
\xi
=
z+\Psi_{a,U_a}(z)+h,
\
z=P_{\mathcal Z_{a,U_a}}\xi,
\
h\in\mathcal H_{a,U_a},
\]
where
\(h=P_{\mathcal H_{a,U_a}}\xi-\Psi_{a,U_a}(z)\) and
\(P_{\mathcal H_{a,U_a}}\xi:=\xi-P_{\mathcal Z_{a,U_a}}\xi\) for
\(\xi\in\mathcal X_{a,U_a}\).
Moreover, by \eqref{eqn:Psi-bound},
$$
\|h\|_{\#,a}
\le
C\|\xi\|_{\#,a}+C|z|
\le
C\|\xi\|_{\#,a}\le
C d_{\#,a}(v).
$$
Consequently,
\begin{equation*}
|z|+\|h\|_{\#,a}
\le
C d_{\#,a}(v),
\end{equation*}
which means that $|z|+\|h\|_{\#,a}$  is sufficiently small.
Thus,
using \eqref{eq:sec23-074}, \eqref{eq:sec23-138} and \eqref{eqn:Psi-bound}, we infer that
\[
D^2\mathscr E_{a,U_a}
\bigl(z+\Psi_{a,U_a}(z)+th\bigr)
[h,h]
\ge
c\delta_a^{-2}\|h\|_{\#,a}^2
\]
for every \(t\in[0,1]\).
By definition of \(\Psi_{a,U_a}\),
\[
D\mathscr E_{a,U_a}
\bigl(z+\Psi_{a,U_a}(z)\bigr)[h]
=
0.
\]
Applying Taylor's formula
to $\mathscr E_{a,U_a}
\bigl(z+\Psi_{a,U_a}(z)+th\bigr)$, we obtain from the definition of $f_{a,U_a}(z)$ that
\begin{equation}
\label{eq:sec23-156}
\mathscr E_{a,U_a}(\xi)
-
f_{a,U_a}(z)
\ge
c\delta_a^{-2}\|h\|_{\#,a}^2.
\end{equation}
Combining \eqref{eq:sec23-152} and \eqref{eq:sec23-156},
\begin{equation}
\label{eq:sec23-157}
E_a^\infty(\mathcal C_{U_a}(\xi))-e_\infty(a)
\ge
c\delta_a^2|z|^2
+
c\delta_a^{-2}\|h\|_{\#,a}^2.
\end{equation}

On the other hand, since \(U_a\in\mathcal G_a\), we have
$$
\operatorname{dist}_{*,a}
\left(
\mathcal C_{U_a}(\xi),\mathcal G_a
\right)
\le
\left\|
\mathcal C_{U_a}(\xi)-U_a
\right\|_{*,a}.
$$
By \(\|U_a\|_{L^2(\mathbb R^2)}=1\) and
\(\langle U_a,\xi\rangle=0\), there holds
$
\|U_a+\xi\|_{L^2(\mathbb R^2)}^2
=
1+\|\xi\|_{L^2(\mathbb R^2)}^2.
$
Set
$$
m_\xi
:=
\|U_a+\xi\|_{L^2(\mathbb R^2)}
=
\bigl(1+\|\xi\|_{L^2(\mathbb R^2)}^2\bigr)^{1/2}.
$$
Recall that $\mathcal C_{U_a}(\xi)
=\frac{U_a+\xi}{m_\xi},$
 then
$$
\mathcal C_{U_a}(\xi)-U_a
=
\frac{\xi}{m_\xi}
+
\left(\frac1{m_\xi}-1\right)U_a.
$$
Since \(m_\xi\ge1\),
$$
\left|\frac1{m_\xi}-1\right|
=
\frac{m_\xi-1}{m_\xi}
=
\frac{\|\xi\|_{L^2(\mathbb R^2)}^2}
{m_\xi(m_\xi+1)}
\le
\frac12\|\xi\|_{L^2(\mathbb R^2)}^2
\le
\frac12\|\xi\|_{*,a}^2.
$$
 Therefore, by
the norm equivalence \eqref{eq:sec23-131} and
\eqref{eqn:Psi-bound}, there is a constant $C>0$ such that
\begin{equation*}
\label{eq:mass-chart-distance}
\begin{aligned}
\left\|
\mathcal C_{U_a}(\xi)-U_a
\right\|_{*,a}
&\le
\frac1{m_\xi}\|\xi\|_{*,a}
+
\left|\frac1{m_\xi}-1\right|
\|U_a\|_{*,a}
\\
&\le
\|\xi\|_{*,a}
+
2\|\xi\|_{*,a}^2
\\
&\le
C(|z|+\|h\|_{*,a}),
\end{aligned}
\end{equation*}
which yields
$$
\operatorname{dist}_{*,a}
\left(
\mathcal C_{U_a}(\xi),\mathcal G_a
\right)
\le
C\left(
|z|+\|h\|_{*,a}
\right).
$$
Note that \(\delta_a<1\) and
\(\|h\|_{*,a}\le\|h\|_{\#,a}\), then \eqref{eq:sec23-157} implies that
\begin{equation*}
E_a^\infty(\mathcal C_{U_a}(\xi))-e_\infty(a)
\ge
c\delta_a^2
\operatorname{dist}_{*,a}
\left(
\mathcal C_{U_a}(\xi),\mathcal G_a
\right)^2,
\end{equation*}
that is,
\begin{equation}
\label{eq:sec23-159}
E_a^\infty(v)-e_\infty(a)
\ge
c\delta_a^2
\operatorname{dist}_{*,a}
\left(
v,\mathcal G_a
\right)^2,
\end{equation}
if $\operatorname{dist}_{\#,a}(v,\mathcal G_a)$ is sufficiently small.

\medskip
\noindent
\textbf{Step 8. From the \(\#\)-norm to the stated \(*\)-norm.}

This final step is needed because smallness of
\(\operatorname{dist}_{*,a}\) does not imply smallness of $\operatorname{dist}_{\#,a}$. Let
\[
d_a(v)
:=
\operatorname{dist}_{*,a}(v,\mathcal G_a).
\]
We may assume that \(d_a(v)>0\) and choose
\(U_a\in\mathcal G_a\) such that
\[
\|v-U_a\|_{*,a}
\le
2d_a(v).
\]
Since \(\|v\|_{L^2(\mathbb R^2)}=\|U_a\|_{L^2(\mathbb R^2)}=1\),
\begin{equation}
\label{eq:sec23-160}
2\langle U_a,v-U_a\rangle
+
\|v-U_a\|_{L^2(\mathbb R^2)}^2
=0.
\end{equation}
Expanding the energy and using the Euler-Lagrange equation of \(U_a\),
together with \eqref{eq:sec23-160}, give
\begin{equation}
\label{eq:sec23-161}
\begin{aligned}
E_a^\infty(v)-e_\infty(a)
&=
\|\nabla(v-U_a)\|_{L^2(\mathbb R^2)}^2
+
\int_{\mathbb R^2}V(x)(v-U_a)^2\,\mathrm dx
-
\mu_a^\infty\|v-U_a\|_{L^2(\mathbb R^2)}^2
\\
&\quad
-
3a\int_{\mathbb R^2}U_a^2(x)(v-U_a)^2\,\mathrm dx
-
2a\int_{\mathbb R^2}U_a(x)(v-U_a)^3\,\mathrm dx
\\
&\quad
-
\frac a2\int_{\mathbb R^2}(v-U_a)^4\,\mathrm dx.
\end{aligned}
\end{equation}
Then \eqref{eq:sec23-103}-\eqref{eq:sec23-104} and H\"older's inequality imply
\begin{equation}
\label{eq:sec23-162}
\delta_a^2
\left|
\int_{\mathbb R^2}U_a(x)(v-U_a)^3\,\mathrm dx
\right|
\le
C\|v-U_a\|_{*,a}^3\
\text{and}\
\delta_a^2
\int_{\mathbb R^2}(v-U_a)^4\,\mathrm dx
\le
C\|v-U_a\|_{*,a}^4.
\end{equation}
Choose  a small constant \(r_\#>0\) such that
\eqref{eq:sec23-159} is valid
for $\operatorname{dist}_{\#,a}(v,\mathcal G_a)<r_{\#}$.
Let \(\rho_0>0\) be so small that $2\rho_0<\frac{r_\#}{2}$
and
\begin{equation}
\label{eq:sec23-163}
C\left(
(2\rho_0)^2
+
(2\rho_0)^3
+
(2\rho_0)^4
\right)
\le
\frac{r_\#^2}{4}.
\end{equation}
Assume \(d_a(v)<\rho_0\), then $\|v-U_a\|_{*,a}<2\rho_0.$
If $\|v-U_a\|_{\#,a}<r_\#,$
the estimate \eqref{eq:sec23-159} is available and \eqref{eq:sec23-075} holds. If
$\|v-U_a\|_{\#,a}\ge r_\#,$
then
\[
\delta_a^2
\int_{\mathbb R^2}V(x)(v-U_a)^2\,\mathrm dx
=
\|v-U_a\|_{\#,a}^2
-
\|v-U_a\|_{*,a}^2
\ge
\frac12r_\#^2.
\]
Recall that \(\mu_a^\infty<0\) for \(a\) close to \(a^*\). Multiplying
\eqref{eq:sec23-161} by \(\delta_a^2\) and using
\eqref{eq:sec23-104} and \eqref{eq:sec23-162}, we obtain
\[
\begin{aligned}
\delta_a^2\bigl[E_a^\infty(v)-e_\infty(a)\bigr]
&\ge
\delta_a^2
\int_{\mathbb R^2}V(x)(v-U_a)^2\,\mathrm dx
\\
&\quad
-
C\left(
\|v-U_a\|_{*,a}^2
+
\|v-U_a\|_{*,a}^3
+
\|v-U_a\|_{*,a}^4
\right).
\end{aligned}
\]
By the above lower bound for the potential term and
\eqref{eq:sec23-163}, the right-hand side is at least
\(\frac{r_\#^2}{4}\). Hence
\begin{equation}
\label{eq:sec23-165}
E_a^\infty(v)-e_\infty(a)
\ge
\frac{r_\#^2}{4}\delta_a^{-2}\ge c\delta_a^2d_a(v)^2
\end{equation}
for all \(a\) sufficiently close to \(a^*\). Thus
\eqref{eq:sec23-075} holds in both cases, and the proof is complete.
\end{proof}

\section{Exponential estimates for the energy gap}\label{sec:boundary-correction}

 In this section, we compare the bounded domain and whole space minimal energies.
Define
\begin{equation*}
\Delta_\Om(a)
:=
e_\Om(a)-e_\infty(a).
\end{equation*}
Since every $H_0^1(\Om)$ function extends by zero to an $H_V^1(\mathbb R^2)$
function,
\[
\Delta_\Om(a)\ge0.
\]
By a standard diagonal argument, we may
assume that all the bounded domain and whole space asymptotic results established in Sections \ref{sec:Blow-up}-\ref{sec:Nondegeneracy}  and used below hold along a common subsequence
\(\{a_k\}\). For simplicity, we also omit the index $k$.

\begin{proposition}\label{prop:exp}
For every $\eta\in(0,1/4)$, there exist $c_\eta,C_\eta>0$ and an integer
$N\ge0$ such that, for \(a\) sufficiently close to $a^*$ along the common subsequence fixed above,
\begin{equation*}
\Delta_\Om(a)
\le
C_\eta\da^{-N}
\exp\left(
-\frac{2(1-\eta)D_\Om}{\da}
\right)
\end{equation*}
and
\begin{equation*}
\Delta_\Om(a)
\ge
c_\eta\da^2
\exp\left(
-\frac{2(1+\eta)d_\Om(x_a)}{\da}
\right),
\end{equation*}
where $x_a$ is the maximum point of the bounded domain minimizer $u_a$.
\end{proposition}

\begin{proof}
We first prove the upper bound.  Choose $y_*\in\mathcal M_\Om$ and rotate a
whole space minimizer $U_a$ so that its maximum point $p_a$ has the same
angular direction as $y_*$.  Then $p_a-y_*=o(\da).$
Choose radii
\[
r_1=\left(1-\frac\eta4\right)D_\Om,
\
r_2=\left(1-\frac\eta8\right)D_\Om.
\]
Since \(d_\Om(y_*)=D_\Om\), one has
\(B_{r_2}(y_*)\subset\Om\). Choose a cut-off
\(\chi\in C_c^\infty(B_{r_2}(y_*))\) with
\(\chi\equiv1\) on \(B_{r_1}(y_*)\). On the set
\(\{x\in\mathbb R^2:\chi(x)\ne1\}\),
\begin{equation}\label{eqn:x-p-a-lower}
|x-p_a|
\ge r_1-o(\delta_a)
=
\left(1-\frac\eta4\right)D_\Om-o(\delta_a)
\end{equation}
as $a\nearrow a^*$.
Set
\[
v_a:=\chi U_a,
\
m_a:=\|v_a\|_{L^2(\mathbb R^2)}^2,
\
\widehat v_a:=m_a^{-1/2}v_a.
\]
Then \(\widehat v_a\in H_0^1(\Om)\) and
\(\|\widehat v_a\|_{L^2(\Om)}=1\).
Apply the exponential upper bound in Lemma~\ref{lem:whole} with the smaller
parameter \(\eta/4\). Replacing $C_\eta$ by a larger constant if necessary, we have
\[
U_a(x)
\le
C_\eta\delta_a^{-1}
\exp\left(
-\frac{(1-\eta/4)|x-p_a|}{\delta_a}
\right)
\]
and
\[
|\nabla U_a(x)|
\le
C_\eta\delta_a^{-2}
\exp\left(
-\frac{(1-\eta/4)|x-p_a|}{\delta_a}
\right)
\]
for
$|x-p_a|\ge R_\eta\delta_a$ and \(a\) sufficiently close to \(a^*\).
Since $(1-\eta/4)^2>1-\eta$ and \(V(x)=(|x|-1)^2\) has at most quadratic growth, then by
\eqref{eqn:x-p-a-lower}, we conclude that there is
an integer \(N\ge0\) such that
for \(a\) sufficiently close to \(a^*\),
\begin{equation*}
1-m_a
\le
C_\eta\da^{-N}
e^{-2(1-\eta)D_\Om/\da}
\end{equation*}
and
\begin{equation*}
\int_{\{x\in\mathbb R^2:\chi(x)\neq1\}}
\left(
U_a^2(x)+|\nabla U_a(x)|^2+U_a^4(x)+V(x)U_a^2(x)
\right)\mathrm dx
\le
C_\eta\da^{-N}
e^{-2(1-\eta)D_\Om/\da}.
\end{equation*}
Expanding
\[
|\nabla(\chi U_a)|^2
=
\chi^2|\nabla U_a|^2
+
U_a^2|\nabla\chi|^2
+
2\chi U_a\nabla\chi\cdot\nabla U_a
\]
and using
\[
2|\chi U_a\nabla\chi\cdot\nabla U_a|
\le
\chi^2|\nabla U_a|^2
+
|\nabla\chi|^2U_a^2,
\]
we obtain
\[
E_a^\Om(\widehat v_a)-e_\infty(a)
\le
C_\eta\da^{-N}
e^{-2(1-\eta)D_\Om/\da}.
\]
Hence the upper bound follows.

For the lower bound,
let \(\widetilde u_a\) be the zero extension of the bounded domain
minimizer \(u_a\). Then
\[
\Delta_\Omega(a)
=
E_a^\infty(\widetilde u_a)-e_\infty(a).
\]
Recall that, throughout the proof,
\(a\) is taken sufficiently
close to \(a^*\) along the common subsequence fixed in Sections \ref{sec:Blow-up}-\ref{sec:Nondegeneracy}.
Let \(U_a\in\mathcal G_a\) be the whole space minimizer chosen along
this subsequence, and let \(p_a\) be the maximum point. By Lemma \ref{lem:whole},
\[
W_a(y)
:=
\delta_a U_a(p_a+\delta_a y)
\to q(y)
\
\text{ in }H^1(\mathbb R^2),
\]
and
$
\frac{|p_a|-1}{\delta_a}\to0
$
as $a\nearrow a^*$.
By rotational invariance of the whole space problem,
choose a rotation \(R_a\) such that
$\bar p_a:=R_a p_a$
has the same angular direction as \(x_a\), and define
$\overline U_a(x):=U_a(R_a^{-1}x).$
Then \(\overline U_a\in\mathcal G_a\), \(\overline p_a\) is a maximum point of
\(\overline U_a\) and
$|\overline p_a|=|p_a|.$
Define
\[
\overline W_a(y)
:=
\delta_a\overline U_a(\overline p_a+\delta_a y).
\]
A direct computation gives
$\overline W_a(y)
=W_a(R_a^{-1}y).$
Since \(q\) is radial and the \(H^1(\mathbb R^2)\)-norm is invariant
under rotations, as $a\nearrow a^*$,
\[
\begin{aligned}
\|\overline W_a-q\|_{H^1(\mathbb R^2)}
&=
\|W_a\circ R_a^{-1}
-q\circ R_a^{-1}\|_{H^1(\mathbb R^2)}
\\
&=
\|W_a-q\|_{H^1(\mathbb R^2)}
\to0.
\end{aligned}
\]
Thus $\overline W_a
\to q
\ \text{ in }H^1(\mathbb R^2)$ as $a\nearrow a^*$.
Moreover,
\[
\frac{|\overline p_a|-1}{\delta_a}
=
\frac{|p_a|-1}{\delta_a}
\to0
\]
as $a\nearrow a^*$.
Now we show that
$
\|\widetilde u_a-\overline U_a\|_{*,a}
\to0
$ as $a\nearrow a^*$.
Setting
\[
\widetilde w_a(y)
:=
\delta_a\widetilde u_a(x_a+\delta_a y),
\]
we have $\widetilde w_a(y)
\to q(y)
\
\text{ in }H^1(\mathbb R^2)$ as $a\nearrow a^*$.
Since \(\overline p_a\) and \(x_a\) have the same angular direction,
\[
|\overline p_a-x_a|
=
\bigl||\overline p_a|-|x_a|\bigr|,
\]
which, together with
\[
\frac{|\overline p_a|-1}{\delta_a}\to0\ \text{and}\
\
\frac{|x_a|-1}{\delta_a}\to0\ \text{as}\ a\nearrow a^*,
\]
implies that
\[
\theta_a
:=
\frac{\overline p_a-x_a}{\delta_a}
\to0.
\]
Furthermore,
\[
\delta_a\overline U_a(x_a+\delta_a y)
=
\overline W_a(y-\theta_a).
\]
Note that
$\overline W_a(\,\cdot-\theta_a)
\to q
\ \text{ in }H^1(\mathbb R^2)$
as $a\nearrow a^*$.
Consequently,
\[
\begin{aligned}
\|\widetilde u_a-\overline U_a\|_{*,a}^2
&=
\int_{\mathbb R^2}
|\widetilde u_a-\overline U_a|^2\mathrm dx
+
\delta_a^2
\int_{\mathbb R^2}
|\nabla\widetilde u_a-\nabla\overline U_a|^2\mathrm dx
\\
&=
\left\|
\widetilde w_a-\overline W_a(\,\cdot-\theta_a)
\right\|_{H^1(\mathbb R^2)}^2
\to0
\end{aligned}
\]
as $a\nearrow a^*$.
Hence
$
\operatorname{dist}_{*,a}
(\widetilde u_a,\mathcal G_a)
\to0
$ as $a\nearrow a^*$.
By Lemma~\ref{lem:uniformcoercive}, for \(a\) sufficiently close to $a^*$,
\begin{equation}
\label{eqn:Delta-low}
\Delta_\Omega(a)
\ge
c\delta_a^2
\operatorname{dist}_{*,a}
(\widetilde u_a,\mathcal G_a)^2.
\end{equation}

In the following, we estimate $\operatorname{dist}_{*,a}
(\widetilde u_a,\mathcal G_a)$.
Choose \(U_a^*\in\mathcal G_a\) such that
\[
\|\widetilde u_a-U_a^*\|_{*,a}
\le
2\operatorname{dist}_{*,a}
(\widetilde u_a,\mathcal G_a).
\]
Then
\begin{equation}
\label{eq:Ua-star-close}
\|\widetilde u_a-U_a^*\|_{*,a}
\to0 \ \text{as}\ a\nearrow a^*.
\end{equation}
Since the whole space minimizer is unique up to rotations for \(a\) sufficiently close to $a^*$, there exists a rotation \(S_a\)
such that
\[
U_a^*(x)=\overline U_a(S_a^{-1}x).
\]
Set $p_a^*:=S_a\overline p_a,$
then \(p_a^*\) is a maximum point of \(U_a^*\). Define
\[
W_a^*(y)
:=
\delta_aU_a^*(p_a^*+\delta_a y).
\]
A direct computation gives
$W_a^*(y)
=
\overline W_a(S_a^{-1}y).$
Since \(q\) is radial,
\[
\begin{aligned}
\|W_a^*-q\|_{H^1(\mathbb R^2)}
&=
\|\overline W_a\circ S_a^{-1}
-q\circ S_a^{-1}\|_{H^1(\mathbb R^2)}
\\
&=
\|\overline W_a-q\|_{H^1(\mathbb R^2)}
\to0
\end{aligned}
\]
as $a\nearrow a^*$.
Thus
\begin{equation}
\label{eq:Wa-star-profile}
W_a^*
\to q
\
\text{ in }H^1(\mathbb R^2) \ \text{as}\ a\nearrow a^*.
\end{equation}

We next claim that
\begin{equation}
\label{eq:pa-star-xa}
\frac{|p_a^*-x_a|}{\delta_a}
\to0  \ \text{as}\ a\nearrow a^*.
\end{equation}
Set
\[
\xi_a:=\frac{p_a^*-x_a}{\delta_a}.
\]
Since
$p_a^*+\delta_a(y-\xi_a)
=
x_a+\delta_a y,$
we have
$
W_a^*(y-\xi_a)
=
\delta_aU_a^*(x_a+\delta_a y).
$
Hence, by \eqref{eq:Ua-star-close},
\begin{equation}
\label{eq:profile-close}
\left\|
\widetilde w_a
-
W_a^*(\,\cdot-\xi_a)
\right\|_{H^1(\mathbb R^2)}
\to0  \ \text{as}\ a\nearrow a^*.
\end{equation}
Suppose that \eqref{eq:pa-star-xa} is false. Then, up to a subsequence, there exists
\(\varepsilon_0>0\)
such that $|\xi_a|\ge\varepsilon_0.$
Up to a further subsequence if necessary, either
$
\xi_a\to\xi
\in\mathbb R^2\setminus\{0\}\
\text{or}\
|\xi_a|\to\infty.
$
In the first case, by \eqref{eq:Wa-star-profile},
\[
W_a^*(\,\cdot-\xi_a)
\to
q(\,\cdot-\xi)
\
\text{ in }H^1(\mathbb R^2)
\]
as $a\nearrow a^*$.
Since
$\widetilde w_a
\to q
\
\text{ in }H^1(\mathbb R^2)$
as $a\nearrow a^*$,
we obtain
\[
\left\|
\widetilde w_a-W_a^*(\,\cdot-\xi_a)
\right\|_{H^1(\mathbb R^2)}
\to
\|q-q(\,\cdot-\xi)\|_{H^1(\mathbb R^2)}>0  \ \text{as}\ a\nearrow a^*,
\]
contradicting \eqref{eq:profile-close}.
In the second case, fix \(R>0\). Then
\[
\begin{aligned}
\|W_a^*(\,\cdot-\xi_a)\|_{L^2(B_R(0))}
&\le
\|W_a^*-q\|_{L^2(\mathbb R^2)}
+
\|q(\,\cdot-\xi_a)\|_{L^2(B_R(0))}
\\
&\to0
\end{aligned}
\]
as $a\nearrow a^*$.
On the other hand,
$
\widetilde w_a
\to q
\
\text{ in }L^2(B_R(0))
$
as $a\nearrow a^*$.
Choosing \(R>0\) such that
$\|q\|_{L^2(B_R(0))}>0,$
we obtain
\[
\left\|
\widetilde w_a-W_a^*(\,\cdot-\xi_a)
\right\|_{L^2(B_R(0))}
\to
\|q\|_{L^2(B_R(0))}>0  \ \text{as}\ a\nearrow a^*,
\]
again contradicting \eqref{eq:profile-close}. Therefore
\eqref{eq:pa-star-xa} holds.

Since \(d_\Omega\) is \(1\)-Lipschitz, $|d_\Omega(p_a^*)-d_\Omega(x_a)|
\le
|p_a^*-x_a|
=
o(\delta_a)  \ \text{as}\ a\nearrow a^*.$
Thus
\begin{equation}
\label{eq:distance-pa-star}
d_\Omega(p_a^*)
=
d_\Omega(x_a)+o(\delta_a)  \ \text{as}\ a\nearrow a^*.
\end{equation}
Recall that
$\frac{d_\Omega(x_a)}{\delta_a}
\to\infty$ as $a\nearrow a^*$.
Together with \eqref{eq:pa-star-xa}, this implies $p_a^*\in\Omega$
for \(a\) sufficiently close to $a^*$.
Since $\widetilde u_a=0
\ \text{on} \ \Omega^c$,
we have
\begin{equation}
\label{eqn:dist-low}
\begin{aligned}
4\operatorname{dist}_{*,a}
(\widetilde u_a,\mathcal G_a)^2
&\ge
\|\widetilde u_a-U_a^*\|_{*,a}^2
\\
&\ge
\int_{\Omega^c}(U_a^*)^2\mathrm dx.
\end{aligned}
\end{equation}
Let \(z_a\in\partial\Omega\) satisfy
\[
|p_a^*-z_a|
=
d_\Omega(p_a^*).
\]
Since \(\partial\Omega\) is \(C^2\), \(\Omega\) satisfies a uniform
exterior ball condition. Hence there exists \(r_{\rm ext}>0\) such that
\[
B_{r_{\rm ext}}
\bigl(z+r_{\rm ext}\nu(z)\bigr)
\subset\Omega^c
\
(z\in\partial\Omega),
\]
where \(\nu(z)\) denotes the outward unit normal to \(\partial\Omega\).
Fix \(\kappa>0\). For \(a\) sufficiently close to $a^*$,
$
3\kappa\delta_a<r_{\rm ext},
$
and hence
\[
B_{\kappa\delta_a}
\bigl(z_a+2\kappa\delta_a\nu(z_a)\bigr)
\subset\Omega^c.
\]
For $x\in B_{\kappa\delta_a}
\bigl(z_a+2\kappa\delta_a\nu(z_a)\bigr)$,
\begin{equation}
\label{eqn:exterior-ball-est}
d_\Omega(p_a^*)-3\kappa\delta_a
\le
|x-p_a^*|
\le
d_\Omega(p_a^*)+3\kappa\delta_a.
\end{equation}
Note that
$\frac{d_\Omega(p_a^*)}{\delta_a}
\to\infty  \ \text{as}\ a\nearrow a^*,$
then there exists a fixed \(L>0\) such that, for \(a\) sufficiently close to $a^*$ and  $x\in B_{\kappa\delta_a}
\bigl(z_a+2\kappa\delta_a\nu(z_a)\bigr)$,
\[
R_\eta\delta_a
\le
|x-p_a^*|
\le L.
\]
Applying the lower exponential estimate in Lemma~\ref{lem:whole}, we deduce that
\[
U_a^*(x)
\ge
c_{\eta,L}\delta_a^{-1}
\exp\left(
-\frac{(1+\eta)|x-p_a^*|}{\delta_a}
\right).
\]
Since the  ball $B_{\kappa\delta_a}
\bigl(z_a+2\kappa\delta_a\nu(z_a)\bigr)$ has area
\(\pi\kappa^2\delta_a^2\), it follows from
\eqref{eqn:exterior-ball-est} that
\[
\begin{aligned}
\int_{\Omega^c}(U_a^*)^2\mathrm dx
&\ge
\pi\kappa^2c_{\eta,L}^2
e^{-6\kappa(1+\eta)}
\exp\left(
-\frac{2(1+\eta)d_\Omega(p_a^*)}{\delta_a}
\right).
\end{aligned}
\]
Using \eqref{eq:distance-pa-star}, we conclude that
for \(a\) sufficiently close to $a^*$,
there is $c_\eta>0$,
\[
\int_{\Omega^c}(U_a^*)^2\mathrm dx
\ge
c_\eta
\exp\left(
-\frac{2(1+\eta)d_\Omega(x_a)}{\delta_a}
\right).
\]
Combining this with \eqref{eqn:Delta-low} and
\eqref{eqn:dist-low}, we finally obtain for \(a\) sufficiently close to $a^*$,
\[
\Delta_\Omega(a)
\ge
c_\eta\delta_a^2
\exp\left(
-\frac{2(1+\eta)d_\Omega(x_a)}{\delta_a}
\right).
\]
\end{proof}
\section{Proof of the main results}\label{sec:Proof-main}

In this section, we give the proof of Theorem~\ref{thm:main}.
 Theorem~\ref{thm:main-1} follows directly from
Lemmas~\ref{lem:exist}-\ref{lem:scale}.

\begin{proof}[Proof of Theorem~{\rm\ref{thm:main}}]
Choose a subsequence of $\{a_k\}$, still denoted by
$\{a_k\}$, along
which all the results of Sections \ref{sec:Blow-up}-\ref{sec:boundary-correction}
 hold. We
 prove that
\[
d_\Om(x_{a_k})\to D_\Om\ \text{as}\ k\to\infty.
\]
By Proposition~\ref{prop:exp},
\begin{equation*}
c_\eta\delta_{{a_k}}^2
e^{-2(1+\eta)d_\Om(x_{a_k})/\delta_{a_k}}
\le
\Delta_\Om(a_k)
\le
C_\eta\delta_{a_k}^{-N}
e^{-2(1-\eta)D_\Om/\delta_{a_k}}
\end{equation*}
for $k$ sufficiently large.
Moreover, Lemma~\ref{lem:scale} gives
$x_{a_k}\to x_0\in\Gamma$ as $k\to\infty$. Hence,
\[
\lim_{k\to\infty} d_\Om(x_{a_k})=d_\Om(x_0)\le D_\Om.
\]
Suppose now that $d_\Om(x_{a_k})\not\to D_\Om$.  Then there exist
$\sigma>0$ and a subsequence such that
\[
d_\Om(x_{a_k})\le D_\Om-\sigma.
\]
Choose $\eta>0$ so small that
\[
(1+\eta)(D_\Om-\sigma)<(1-\eta)D_\Om.
\]
Then
\[
\frac{c_\eta}{C_\eta}
\delta_{a_k}^{N+2}
\exp\left\{
\frac{
2[(1-\eta)D_\Om-(1+\eta)(D_\Om-\sigma)]
}{\delta_{a_k}}
\right\}
\le1,
\]
whereas the left-hand side tends to $+\infty$, a contradiction.  Thus
\[
d_\Om(x_{a_k})\to D_\Om\ \text{as}\ k\to\infty
\]
and
$d_\Om(x_0)=D_\Om$ with $x_0\in\mathcal M_\Om$.  Therefore,
$
\dist(x_{a_k},\mathcal M_\Om)\to0
$
as $k\to\infty$.
\end{proof}
\ \\
\noindent\textbf{\Large Acknowledgements}

\noindent {This research is supported by National Natural Science Foundation of China (No.12371120).}
\vspace{0.8em}

\noindent\textbf{\Large Declaration}

\noindent {{\bf Conflict of interest} The authors do not have conflict of interest.}\\

\end{document}